\documentclass[twoside,12pt]{article}

\usepackage{amsmath}
\usepackage{amssymb}
\usepackage{pxfonts}
\usepackage{graphicx}
\usepackage{eucal}
\usepackage{mathrsfs}
\usepackage{theorem}
\usepackage{pifont}
\usepackage{amscd}

\usepackage{sectsty}

\usepackage{color}
\usepackage{fancyhdr}
\usepackage{framed}
\usepackage[all]{xy}
\usepackage{hyperref}
\usepackage{dsfont}

\definecolor{shadecolor}{rgb}{0.8,0.8,0.8}

\theoremheaderfont{\fontfamily{pzc}\bfseries\large}
\newtheorem{theorem}{Theorem}[section]

\newtheorem{conjecture}[theorem]{Conjecture}
\newtheorem{lemma}[theorem]{Lemma}
\newtheorem{proposition}[theorem]{Proposition}
\newtheorem{corollary}[theorem]{Corollary}
\newtheorem{definition}[theorem]{Definition}

\usepackage[all]{xy}
\usepackage{tikz} 
\usepackage{tkz-euclide}
\usetkzobj{all}
\usepackage{pgfplots}

\newcommand{\btkz}{\begin{tikzpicture}}
\newcommand{\etkz}{\end{tikzpicture}}

\pgfdeclareradialshading{ballshading}{
 \pgfpoint{-10bp}{10bp}}
 {color(0bp)=(gray!10!white); 
  color(9bp)=(gray!35!white);
  color(18bp)=(gray!30!black!40); 
  color(25bp)=(gray!20!black!40); 
  color(50bp)=(black)}

{\theorembodyfont{\rmfamily}

\newenvironment{proof}{{\flushleft \emph{Proof}:}}{\hfill\ding{110}}
\newenvironment{comment}{{\flushleft \fontfamily{pzc}\bfseries\large Comment:}}{}

\newcommand{\secref}[1]{Section~\ref{#1}}

\newcommand{\defref}[1]{Definition~\ref{#1}}

\newcommand{\propref}[1]{Proposition~\ref{#1}}
\newcommand{\lemref}[1]{Lemma~\ref{#1}}
\newcommand{\corref}[1]{Corollary~\ref{#1}}

\newcommand{\Cof}{\operatorname{Cof}}

\newcommand{\GL}{\operatorname{GL}}
\newcommand{\End}{\operatorname{End}}

\newcommand{\Vol}{\text{Vol}}

\newcommand{\N}{\mathcal{N}}
\newcommand{\sig}{\sigma}
\newcommand{\M}{\mathcal{M}}

\newcommand{\image}{\operatorname{Image}}
\newcommand{\rank}{\operatorname{rank}}

\newcommand{\Span}{\operatorname{span}}

\newcommand{\id}{\operatorname{Id}}
\newcommand{\g}{\mathfrak{g}}
\newcommand{\h}{\mathfrak{h}}
\newcommand{\sgn}{\text{sgn}}
\newcommand{\R}{\mathbb{R}}

\renewcommand{\M}{\mathcal{M}}

\newcommand{\til}{\tilde}

\newcommand{\brk}[1]{\left(#1\right)}          
\newcommand{\Average}[1]{\left\langle#1\right\rangle}      

\newcommand{\beq}{\begin{equation}}
\newcommand{\eeq}{\end{equation}}

\newcommand{\IP}[2]{\Average{#1,#2}}

\newcommand{\de}{\delta}
\newcommand{\ep}{\epsilon}

\makeatletter
\newcommand{\extp}{\@ifnextchar^\@extp{\@extp^{\,}}}
\def\@extp^#1{\mathop{\bigwedge\nolimits^{\!#1}}}
\makeatother

\numberwithin{equation}{section}

\begin{document}
\title{Regularity via minors and applications to conformal maps}

\author{Asaf Shachar\footnote{Institute of Mathematics, The Hebrew University of Jerusalem.} }

\date{}
\maketitle

\begin{abstract}
We prove that if the minors of degree $k$ of a Sobolev map $\R^d \to \R^d$ are smooth then the map is smooth, when $k,d$ are not both even. We use this result to derive a simple, self-contained proof of the famous Liouville theorem for conformal maps, under the weakest possible regularity assumptions, in even dimensions which are not multiple of $4$. This is based on the approach taken in \cite{Iwa93} by Iwaniec and Martin. We also prove the regularity of $W^{1,d/2}$ conformal maps between Riemannian manifolds, under the additional assumption of continuity.
\end{abstract}


\tableofcontents


\section{Introduction}
\subsection{Regularity via minors}


Let $d>2$, and let $\Omega$ be an open subset of $\R^d$. Let $2 \le k \le d-1$ be a fixed integer  and let $p \ge 1$. Consider the following question:
\begin{quote}
\emph{
Let $f \in W^{1,p}(\Omega,\mathbb{R}^d)$ satisfy either $\det df>0$ a.e.~ or $\det df<0$ a.e. Suppose that $\extp^k df  \in \End(\extp^k \R^d)$ is smooth. Is $f$ smooth?
}
\end{quote}


Here $\extp^k \R^d$ is the $k$-th exterior algebra of $\R^d$, and for a linear map $A:\R^d \to \R^d$, $\extp^k A$ is its $k$-th exterior power, i.e. the endomorphism of $\extp^k \R^d$ defined via the formula
\[
v_1 \wedge \dots \wedge v_k \to Av_1 \wedge \dots \wedge Av_k.
\]
$\extp^k A$ encodes all the $k$-minors of the map $A$. For $f \in W^{1,p}(\Omega,\mathbb{R}^d)$, $\extp^k df$ is a map $\Omega \to \End(\extp^k \R^d)$, which is defined a.e. on $\Omega$. Our assumption is that this map has a smooth representative. (Equivalently, the map $\Omega \to \R^{\binom{d}{k}^2 }$ which takes $x$ to all the $k$-minors of $df_x$ has a smooth representative).

The motivation for studying this question is the following: There are situations in which one has accessibility  only to information about the regularity of the $k$-minors of the differential of a Sobolev map, for some specific value of $k$. It is then natural to examine what can be said about maps having regular minors. In particular, this question arises in the context of regularity of weakly conformal maps. We give a partial answer to the question, namely:
\begin{theorem}[Regularity via minors-invertible case]
\label{thm:regularity_via_minors_unique_case}
Let $d>2$ and let $\Omega \subseteq \R^d$ be open. Let $2 \le k \le d-1,p \ge 1$.
Let $f \in W^{1,p}(\Omega,\mathbb{R}^d)$ and suppose that either $\det df>0$ a.e. or $\det df<0$ a.e. and that $\extp^k df \in \GL(\extp^k \R^d)$ is smooth. 
If $k,d$ are not both even, $f$ is smooth. 
If $k$ is odd, the assumption on $\text{sgn} (\det df)$ can be omitted. 
\end{theorem}
\begin{comment}
The assumption $\extp^k df \in \GL(\extp^k \R^d)$ implies that $df \in  \GL( \R^d)$ a.e., so we always have $\det df \neq 0$ a.e., even when we do not assume anything on $\text{sgn} (\det df)$. The requirement  $\det df>0$ a.e. (or its weak counterpart $\det df \ge 0$ a.e.) is very common in the regularity theory of mappings (quasiregular, quasiconformal, of bounded distortion, etc.). Without it, many of the results are no longer valid.
\end{comment}


To learn something on a map from its $k$-minors, some non-degeneracy conditions must be assumed; indeed, if $f$ is a Sobolev map with $\rank (df)<k$ a.e. then the $k$-minors vanish identically, hence do not provide any information on $f$. Throughout the paper, we will examine  conditions that enable deducing properties of a map from properties of its minors. It turns out that for the sake of applications, the invertibility assumption $\extp^k df \in \GL(\extp^k \R^d)$  is too restrictive; after omitting it, we get
the following corollary of Theorem \ref{thm:regularity_via_minors_unique_case}:

\begin{theorem}[Regularity via minors]
\label{thm:regularity_via_minors_unique_case_a.e._assumptions}
Let $d>2$ and let $\Omega \subseteq \R^d$ be open. Let $2 \le k \le d-1,p \ge 1$. Let $f \in W^{1,p}(\Omega,\mathbb{R}^d)$ and suppose that either $\det df>0$ a.e. or $\det df<0$ a.e. and that $\extp^k df \in \End(\extp^k \R^d)$ is smooth. 
Then, if $k,d$ are not both even, $f$ is smooth on an open subset of full measure in $\Omega$. 
If $k$ is odd, it suffices to assume $\det df \neq 0$ a.e.. 
\end{theorem}

This version follows immediately from Theorem \ref{thm:regularity_via_minors_unique_case}: Indeed, for every linear map  $A:\R^d \to \R^d$, $A$ is invertible if and only if $\extp^k A$ is invertible. (this follows from the fact that the map $A \to \extp^k A$ is multiplicative, i.e. $\extp^k A \circ B=\extp^k A \circ \extp^k B$, so $(\extp^k A)^{-1}=\extp^k A^{-1}$).

Define $\Omega_0=\{ x \in \Omega \, | \, \bigwedge^k df_x \in \text{GL}(\bigwedge^{k}\mathbb{R}^d) \}$.  
$\Omega_0$ is open, as it is an inverse image of the open subset $\text{GL}(\bigwedge^{k}\mathbb{R}^d) \subseteq \text{End}(\bigwedge^{k}\mathbb{R}^d)$ via the continuous map $x \to  \bigwedge^k df_x$. Since $df_x$ is invertible a.e.   $\bigwedge^k df_x$ is invertible a.e., so $\Omega_0$ has full measure in $\Omega$. Now apply Theorem \ref{thm:regularity_via_minors_unique_case} to $f|_{\Omega_0}$. Other relaxations of the rank assumptions will be studied in \secref{sec:sec_odd_k}.

\paragraph{A proof sketch of Theorem \ref{thm:regularity_via_minors_unique_case} }

Under the given assumptions, the $k$-minors of $df$ uniquely determine $df$ in a smooth way. In other words, the map
\[
\psi:A   \to \extp^k A
\] 
is an embedding. Composing the smooth map $x \to \extp^k df_x$ with $\psi^{-1}$ establishes the smoothness of $f$. In more detail: When $k,d$ are not both even, $\extp^k df$ uniquely determines $df$, since $\sgn(\det df)$ is constant. Indeed, if $A,B \in \GL^+(\mathbb{R}^d)$ or $A,B \in \GL^-(\mathbb{R}^d)$ and $\extp^k A=\extp^k B$, then $A=B$: To see this write $S=AB^{-1}$. Then $\extp^k S=\text{Id}_{\extp^k \mathbb{R}^d}$ which implies $S=\pm \id$. The assumptions on $k,d$ force $S=\id$.  The delicate part is to understand in what sense $\psi: A \mapsto \extp^k A$ is a diffeomorphism; considering $\psi$ as a map $\GL^+(\R^d) \to \GL(\extp^k \R^d)$, we shall show it is a smooth embedding, so $\psi:\GL^+(\R^d) \to \text{Image} (\psi)$ is a diffeomorphism. 

\subsection{Regularity of weakly conformal maps (Euclidean case)}
\label{sec:intro_Regularity_of_weakly_conformal_maps}
The well-known Liouville theorem states the following:
\begin{quote}
\emph{
Let $f \in W_{loc}^{1,d}(\Omega,\R^d)$ be a weakly-conformal map, which is either orientation-preserving or orientation-reversing almost everywhere, that is
\[
df^T df=|\det df|^{2/d} \id \, \text{ a.e. }
\]
and
\[
\det df \ge 0 \, \, \text{ a.e.} \, \, \text{  or } \,   \det df \le 0 \, \, \text{ a.e. }
\]
Then $f$ is either constant or a restriction to $\Omega$ of a M\"obius transformation of $\mathbb{R}^d$. In particular, $f$ is smooth.
}
\end{quote}
This theorem has a long history, and was originally proved under much stronger regularity assumptions; There are various proofs, none of them trivial; see e.g. \cite{geh62rings,BI82,Res67b}. For a more thorough survey, and for a presentation of various proofs under stronger regularity assumptions, see \cite{IM02}.

In the seminal paper \cite{Iwa93}, Iwaniec and Martin proved a stronger version of Liouville's theorem under weaker regularity assumptions. They showed that when $d>2$ is even, $W^{1,d/2}$-conformal maps are smooth, and  for every $1<p<d/2$, there exist weakly conformal $W_{loc}^{1,p}$ maps that are not smooth. ($p=d/2$ is the critical exponent for this problem.)

The authors in \cite{Iwa93} proved this stronger version by \emph{reducing} it to the $W^{1,d}$ case, which was already known. Their proof proceeds along the following lines:
They show that every differential form $df^{i_1} \wedge df^{i_2} \wedge \dots \wedge df^{i_{d/2}}$, where $1 \le i_1,\dots,i_{d/2} \le d$, is weakly harmonic, hence smooth. This immediately implies $\det df \in C^{\infty}$, hence (using conformality) $f \in W^{1,d}$. Note that the forms $df^{i_1} \wedge df^{i_2} \wedge \dots \wedge df^{i_{d/2}}$ encode the $d/2$-minors of $df$.

By reverting to the  $W^{1,d}$ version of Liouville's theorem, the authors in \cite{Iwa93} give up a lot of information: 
After establishing the smoothness of all the $d/2$-minors (which constitute $\binom d {d/2}$ polynomial combinations of derivatives of $f$), they only use one specific corollary---that the Jacobian of $f$ is smooth. By doing this, they recover a non-trivial problem, whose proof is not simple. This paper proposes an alternative approach for completing the proof whereby every $W^{1,d/2}$ conformal map is smooth, relying directly on the smoothness of the $d/2$-minors. 

Theorem \ref{thm:regularity_via_minors_unique_case_a.e._assumptions} implies that when $d$ is not a multiple of $4$, every $W^{1,d/2}$ conformal map $f$ is smooth on an open subset of full measure, if we require the slightly stronger assumption $\det df>0$ (compared to $\det df \ge 0$).  A further easy application of the Liouville's theorem, which provides an \emph{explicit formula} for smooth conformal maps, allows us to deduce smoothness on the entire domain.


The reduction made in \cite{Iwa93} is  very elegant, and quite elementary; it exploits the  \emph{conformal symmetry} of the Hodge dual operator in even dimensions, and only uses  standard \emph{linear} elliptic regularity results. The strategy is to replace the non-linear first order equation of conformality, with second-order equations which are "linear in the minors of $df$". By combining this reduction with Theorem \ref{thm:regularity_via_minors_unique_case_a.e._assumptions}, we offer a self-contained regularity proof, which we hope would be more accessible to non-experts. Even though the paper \cite{Iwa93} is not short (53 pages), the part concerning the regularity result is very short. To help bring this part into light, we reproduce in this paper its main argument.

\subsection{Regularity of weakly conformal maps (Riemannian case)}
\label{subsec:intro_conf_reg_Riemm}

The regularity of weakly conformal maps has also been studied extensively in the context of maps between Riemannian manifolds. The first result seems to be \cite{lelong1976geometrical}, where it is proven that any conformal homeomorphism between two $C^{\infty}$ Riemannian manifolds is $C^{\infty}$. In \cite{She82}, \cite{res78dif}, \cite{iwa82regthes}
it is proved that
a $W^{1,d}_{loc}$ \emph{continuous} conformal map between two manifolds with $C^r$  metric tensors is  $C^{r+1}$. This was recently re-proved (using different methods) in \cite{LS14}. 


%



It is worth noting that even though all the previous works explicitly assumed continuity in advance, it has been shown recently that this additional assumption (which wasn't made in the Euclidean cases) is unnecessary when assuming $\det df>0$ a.e. (and when the target manifold is compact): In \cite{goldstein2017finite} it is proven that any $W^{1,d}(\M,\N)$ map of finite distortion is 
continuous, if $\N$ is compact, (A map of finite distortion is a map satisfying $\det df > 0$ or $df=0$ a.e.) 

We improve the known results by proving the regularity of continuous, $W^{1,d/2}$ conformal maps between $d$-dimensional manifolds, for even $d$. More precisely, we prove  the following theorem:




\begin{theorem}
\label{thm:reg_conf_Riemann}
Let $(\M,\g),(\N,\h)$ be smooth oriented $d$-dimensional Riemannian manifolds ($d$ even).
Let $f \in W^{1,\frac{d}{2}}(\M,\N)$ be a weakly conformal \emph{continuous} map. Then $f$ is smooth.
\end{theorem}
For simplicity, we assumed here that the metrics on $\M,\N$ are $C^{\infty} $. We use here the following definition of weakly conformal maps between manifolds:
\begin{definition}
\label{def:weakly_conformal}
We say that $f \in W^{1,s}(\M,\N)$ is weakly conformal, if
\[
df^T df=|\det df|^{2/d} \id_{T\M} \, \text{ a.e. }
\]
and
\[
\det df \ge 0 \, \, \text{ a.e.} \, \, \text{  or } \,   \det df \le 0 \, \, \text{ a.e. }
\]

\begin{comment}
We view $df$ as a map $T\M \to T\N$; for details, see \cite{CV16}. 
\end{comment}

\end{definition}

To the best of our knowledge, this is the first regularity result regarding $W^{1,p}$ weakly conformal maps between manifolds, when $p<d$.

We use the continuity assumption twice; one place is where we pass from the smooth Theorem \ref{thm:conformal_pull_back_pres} to its weak version Theorem  \ref{cor:closed_and_co-closed_pullbacks_preserve}. The derivation of the weak version from the strong one is implicitly based on the fact that $f$ can be approximated \emph{uniformly} in the Sobolev sense via smooth maps.

To understand the  main difference between the Riemannian case and the Euclidean case we briefly describe the key idea behind the Euclidean proof:
\begin{quote}
A weakly conformal map pulls back closed and co-closed forms to weakly  closed and co-closed forms, which are known to be smooth (elliptic regularity). 
\end{quote}
When the target space is Euclidean, i.e. $\N=\R^d$ ($d=2k$), one can choose $\omega=dy^{i_1} \wedge dy^{i_2} \wedge \dots \wedge dy^{i_k}$, where $y^i$ are the standard coordinates on $\R^{2k}$. Clearly, $\omega$ is closed and co-closed, hence $f^*\omega=df^{i_1} \wedge df^{i_2} \wedge \dots \wedge df^{i_k}$ is smooth.
Thus, the $d/2$-minors of $f$ are smooth. When trying to adapt this argument to an arbitrary Riemannian target space, we hit the following problem: A generic Riemannian metric does not admit "higher-order" harmonic coordinates, i.e. there are no local  coordinates $y^i$ on $(\N,\h)$ such that $dy^{1} \wedge dy^{2} \wedge \dots \wedge dy^{k}$ is co-closed. For a proof see \cite{Bryantobstrucharmcoord}.  However, it turns out that all metrics have a weaker property which we shall exploit: the existence of local frames for $\bigwedge^k T^*\N$, whose elements are closed and co-closed forms, for any $1 \le k < d$.



\paragraph{Structure of this paper}
In \secref{sec:sec_main} we prove Theorem  \ref{thm:regularity_via_minors_unique_case}.  In \secref{sec:conformal_reg} we present the reduction of the $W^{1,d/2}$ regularity problem of conformal maps to the $W^{1,d}$ case following \cite{Iwa93}, and explain where Theorem~\ref{thm:regularity_via_minors_unique_case} can be used in order to "skip" the need to prove the  $W^{1,d}$-case separately. In \secref{sec:Regularit_conformalRiemannianmani} we prove the regularity of conformal maps between Riemannian manifolds (Theorem \ref{thm:reg_conf_Riemann} ). In \secref{sec:sec_odd_k} we generalize Theorem \ref{thm:regularity_via_minors_unique_case}, when $k$ is odd; we show that it is possible to smoothly reconstruct $df$ from its $k$-minors, under relaxed rank conditions.

In \secref{sec:sec_discuss} we present some open questions, which arise naturally from this work. 
In particular, we discuss the necessity of the assumption $\bigwedge^k df \in \GL(\bigwedge^k \R^d)$ and the case where $k,d$ are both even, among some other points. In Appendix \ref{section:Local_existence_closed_co-closed_frames} we prove the local existence of closed and co-closed frames of differential forms for arbitrary metrics.
\section{Proof of regularity via minors (Theorem  \ref{thm:regularity_via_minors_unique_case})}
\label{sec:sec_main}

In this section we prove Theorem \ref{thm:regularity_via_minors_unique_case}.  We begin with a few lemmas.


\begin{lemma}
\label{lem:almost_injectivity_of_exterior_map_invertible_case}
Let $V,W$ be $d$-dimensional real vector spaces, and let $1 \le k \le d-1$ be an integer.  If $A,B \in \GL(V,W)$  and $\extp^k A=\extp^k B$, then $A=\pm B$.
\end{lemma}

\begin{proof}
Write $S=AB^{-1}$. Then, $\extp^k S=\text{Id}_{\extp^k W}$. This implies that every $k$-dimensional subspace of $W$ is $S$-invariant. Indeed, let $v_1,\dots,v_k \in W$ be linearly independent. Then $Sv_1 \wedge \dots \wedge Sv_k=v_1 \wedge \dots \wedge v_k$, which implies $\Span\{ v_1,\dots,v_k \}=\Span\{ Sv_1,\dots,Sv_k \}$. By \lemref{lem:k_invariant-subspaces_identity}, $S$ is a multiple of the identity, i.e. $S=\lambda \,\text{Id}$ for some $\lambda \in \R$. Then,
\[
\id_{\extp^k W}=\extp^k S=\lambda^k \id_{\extp^k W} \Rightarrow \lambda^k=1  \Rightarrow \lambda=\pm 1.
\]
\end{proof}

\begin{corollary}
\label{cor:injectivity_of_exterior_map}
Suppose $k,d$ are not both even. If $\extp^k A=\extp^k B$ for $A,B \in \GL^+(\mathbb{R}^d)$ or $A,B \in \GL^-(\mathbb{R}^d)$ , then $A=B$.
\end{corollary}

\begin{proof}
Write $S=AB^{-1}$, hence $\extp^k S=\text{Id}_{\extp^k \mathbb{R}^d}$.
By the assumptions on $A,B$ it follows that $S\in \GL^+(\mathbb{R}^d)$.
By \lemref{lem:almost_injectivity_of_exterior_map_invertible_case},  $S=\pm \id$.
Assume by contradiction that $S=-\id$; $\extp^k S=\text{Id}_{\extp^k \mathbb{R}^d}$ implies that $k$ is even, whereas $S \in \GL^+(\mathbb{R}^d)$ implies that $d$ is even.
\end{proof}

We automatically obtain the following corollary:

\begin{corollary}
\label{cor:exterior_map_properties}
Define $\psi:\GL(\R^d) \to \GL(\extp^k \R^d)$ by $\psi( A)= \extp^k A$. $\psi$ is a smooth locally-injective homomorphism of Lie groups. In particular, $\psi$ is an immersion. When $k$ is odd $\psi$ is injective. When $k$ is even and $d$ is odd, the restrictions $\psi|_{\GL^+(\R^d)},\psi|_{\GL^-(\R^d)}$ are injective.
\end{corollary}

\begin{proof}
All the properties, except for the immersion property are obvious. 
The fact that $\psi$ is an immersion follows from the (elementary) fact that every smooth locally-injective homomorphism of Lie groups is an immersion. One way to see this is as follows: $\psi$ has constant rank, so it must be an immersion by the rank theorem (Theorems 4.12 and 7.5 in \cite{Lee13}). 
\end{proof}

\begin{lemma}
\label{lem:exterior_map_is_proper}
$\psi:\GL(\R^d) \to \GL(\extp^k \R^d)$ is proper. 
\end{lemma}

\begin{corollary}
\label{cor:image_exterior_map_is_closed_embedded}
When $k$ is odd $\psi:\GL(\R^d) \to \GL(\extp^k \R^d)$ is an embedding, and its image $ H:=\text{Image} (\psi)$ is a closed embedded submanifold of $ \GL(\extp^k \R^d)$. When $k$ is even, the same applies to the restrictions $\psi|_{\GL^+(\R^d)},\psi|_{\GL^-(\R^d)}$.
\end{corollary}

\begin{proof}[Of \corref{cor:image_exterior_map_is_closed_embedded}]

Suppose first $k$ is odd. $\psi$ is then a smooth injective immersion by \corref{cor:exterior_map_properties}, and proper by \lemref{lem:exterior_map_is_proper}. 
Every smooth injective proper immersion is an embedding, by elementary topology. (see Proposition 4.22 in \cite{Lee13}). Moreover, every proper continuous map is closed (see Proposition A.57 in \cite{Lee13}). In particular, the image $H:=\text{Image} (\psi)$ is closed. The proof for the case of even $k$ is the same.
\end{proof}



\begin{proof}[Of \lemref{lem:exterior_map_is_proper}]

Let $K \subseteq  \GL(\extp^k \R^d)$ be compact, and let $A_n \in \psi^{-1}(K)$. We shall prove $A_n$ has a convergent subsequence in $\psi^{-1}(K)$. It suffices to prove $A_n$ converges (modulo a subsequence) in $M_d(\R)$; indeed, if $A_n \to A$, then $\extp^k A_n \to \extp^k A$, and the limit $\extp^k A$ must be in $K$. In particular, $\extp^k A \in  \GL(\extp^k \R^d)$ so $A \in \GL(\R^d)$. 




After passing to a subsequence, we may assume $\extp^k A_n$ converges to some element $D \in \GL(\extp^k\mathbb{R}^d)$.

Using singular values decomposition, we may assume that $A_n=\text{diag}(\sigma_1^n,\dots,\sigma_d^n)$  is diagonal (since the orthogonal group is compact, the orthogonal components surely converge after passing to a subsequence).
$\extp^k A_n$ is diagonal with eigenvalues $\Pi_{r=1}^k \sigma_{i_r}^n$, where all the $i_r$ are distinct. So, every such product converges when $n \to \infty$. Let $1\le i \neq j \le d$. Since $k \le d-1$, we can choose some $1 \le i_1,\dots,i_{k-1} \le d$ all different from $i,j$. Since both products 
\[
(\Pi_{r=1}^{k-1} \sigma_{i_r}^n)\sigma_{i}^n \quad\text{and}\quad (\Pi_{r=1}^{k-1} \sigma_{i_r}^n)\sigma_{j}^n
\]
converge to positive numbers (since $D$ is invertible) , so does their ratio $C_{ij}^n=\frac{\sigma_i^n}{\sigma_j^n}$. Now, 
\[
\Pi_{r=1}^k \sigma_{r}^n=\Pi_{r=1}^k \sigma_{1}^n\frac{\sigma_r^n}{\sigma_1^n}=(\sigma_{1}^n)^k  \Pi_{r=1}^k C_{r1}^n
\]
 converges to a positive number. Since all the $C_{r1}^n$ converge, we deduce $\sigma_1^n$ also converges. Without loss of generality, the same holds for every $\sigma_i^n$, so $A_n$ converges. 
\end{proof}

Now, we have all the preliminary results we need for proving the main result:
  
\begin{proof}[Of Theorem \ref{thm:regularity_via_minors_unique_case}]

We prove this claim for the case where $\det df>0$ a.e. The proof of the other two cases, $\det df<0$, and the odd $k$-case is similar.

Recall that $\psi:\GL^+(\R^d) \to \GL(\extp^k \R^d)$ is given by $\psi(A)= \extp^k A$. By \corref{cor:image_exterior_map_is_closed_embedded},  $\psi$ is a  smooth embedding, and $ H:=\text{Image} (\psi)$ is a closed embedded submanifold of $\GL(\extp^k\R^d)$. Thus, $\psi:\GL^+(\R^d) \to H$ is a diffeomorphism.

The map
\[
\phi: x \to \bigwedge^k df_x \,\,,\,\, \phi:\Omega \to \GL\big(\extp^k\R^d\big)
\]
is smooth by assumption, and its image is contained in $H$. Indeed, since $H$ is closed in $ \GL(\extp^k \R^d) $, $\phi^{-1}(H)$ is closed in $\Omega$. Furthermore, on a set of full measure $\det df >0$; for every $x \in \Omega$ where $\det df_x >0$ we clearly have $\phi(x) \in H$. This implies $\phi^{-1}(H)$ is dense in $\Omega$, hence $\phi^{-1}(H)=\Omega$.
 
Since $H$ is an embedded submanifold of $\GL(\extp^k\R^d)$, $\phi$ remains smooth after restricting the codomain to $H$ (see corollary 5.30 in \cite{Lee13}). Thus, the map
\[
\tilde \phi: x \to \bigwedge^k df_x \,\,,\,\, \tilde \phi:\Omega \to H
\]
is smooth.
Finally, since $\psi:\GL^+(\R^d) \to H$ is a diffeomorphism, we deduce that the following map $\Omega \to \GL^+(\mathbb{R}^d)$,
\[
\Omega: x \mapsto \psi^{-1} \circ \tilde \phi(x)=df_x
\]
is smooth. (More precisely, $\psi^{-1} \circ \tilde \phi$ is smooth and coincides with the weak derivative $df$ almost everywhere). This establishes the smoothness of $f$.

Finally, we explain briefly why when $k$ is odd the assumption on the sign of the Jacobian can be omitted.  In that case, $\extp^k df$ uniquely determines $df$, even if we do not know $\sgn(\det df)$. Hence, the map $\GL(\R^d) \to \GL(\extp^k \R^d)$ is invertible, after restricting the codomain to be its image. (see \corref{cor:image_exterior_map_is_closed_embedded}).
\end{proof}

\section{Regularity of conformal maps: The Euclidean case}
\label{sec:conformal_reg}

\subsection{Reduction from $W^{1,d/2}$ to $W^{1,d}$ }
In this section, we show the reduction of Liouville's theorem from the $W^{1,d/2}$ case to the  $W^{1,d}$ case; we essentially present the proof of  \cite{Iwa93} that all the ${d/2}$-minors of $f$ are smooth, phrased in a slightly different language. 

Their proof is based on a clever use of the conformal invariance of the Hodge dual operator. The convenient setting for the proof, which makes it more transparent, is that of Riemannian geometry. (This will also be useful for later use, when we discuss the regularity of conformal maps in the Riemannian setting, in section \ref{sec:Regularit_conformalRiemannianmani}).

%


Let $(\M,\g),(\N,\h)$ be smooth oriented $d$-dimensional Riemannian manifolds. We denote by $\Omega^k(\N)$ the space of smooth differential forms of degree $k$ on $\N$.

Let $\de:\Omega^k(\N) \to \Omega^{k-1}(\N)$ be the adjoint of the exterior derivative; $\de$ is given by the formula 
\[
\de(\omega)=(-1)^{dk+d+1} \star d \star \omega, \, \, \text{ where } \, \, \omega \in \Omega^k(\N).
\]

Let $f:\M \to \N$ be smooth. Here is the idea of the proof: 

Pullback by an arbitrary map commutes with the exterior derivative $d$. Isometries commute with the Hodge dual operator (up to a sign, depending on whether or not they are orientation-preserving). Thus, isometries commute with $\de$, that is $f^*\de \omega=\de f^*\omega$. In particular, pullback by an isometry \emph{preserves co-closedness of forms}, that is $\de \omega=0 \Rightarrow \de f^*\omega=0$.

Now we reach the crucial point of the argument:
 Since the Hodge dual operator acting on forms of degree $d/2$ is conformally invariant, conformal maps also preserve the co-closedness of forms at that degree. (from this perspective, they "behave like" isometries).

In more detail, let $\omega \in \Omega^k(\N)$, and suppose that  $\de \omega=0$, i.e. $d \star \omega =0$. Then 
\beq
\label{eq:Tad_star_conf_intemed0}
d \brk{f^* (\star \omega)}=f^*(d \star \omega)=f^*0=0.
\eeq
Thus, if $\omega$ is co-closed,  then $f^* (\star \omega)$ is closed.
Now, suppose that $f$ is an isometry. Then $f:\M  \to \N$ commutes with the Hodge dual operator, i.e.
\[
f^*(\star \omega)= \pm \star f^*\omega,
\]
and the sign is determined according to whether $f$ is orientation-preserving or orientation-reversing. Writing explicitly the role of the metrics (which determine the Hodge operators),  we see that for \emph{every map} $f$
\[
f^*(\star_{\h} \omega)= \pm \star_{f^*\h} f^*\omega.
\]

In particular, for conformal maps, we have $f^*\h=\lambda \g$ for some $\lambda \in C^{\infty}(\M)$. Now suppose that $d$ is even, and that $k=d/2$. The Hodge dual is conformally invariant in half the dimension, so if  $\omega \in \Omega^\frac{d}{2}(\N)$, then

\beq
\label{eq:Tad_star_conf_intemed1}
f^*(\star_{\h} \omega)= \pm \star_{f^*\h} f^*\omega= \pm \star_{\lambda \g} f^*\omega= \pm \star_{\g} f^*\omega.
\eeq

Note this argument fails if $\lambda=0$ at some point $p$, but then the conformal equation implies $df_p=0$, so both sides of \eqref{eq:Tad_star_conf_intemed1}  vanish at $p$.

If $\omega$ is co-closed, then by combining \eqref{eq:Tad_star_conf_intemed0},\eqref{eq:Tad_star_conf_intemed1} we obtain 

\beq
\label{eq:Tad_star_conf_intemed4}
\delta(f^*\omega)= \pm \star_{\g} d ( \star_{\g}  f^*\omega)=\pm \star_{\g} d f^*(\star_{\h} \omega)=0.
\eeq

%
%
%
%
%


So, we proved the following:
\begin{theorem}
\label{thm:conformal_eq_tadeus}
Let $\M,\N$ be smooth oriented $d$-dimensional Riemannian manifolds ($d$ even), and
let $f:\M \to \N$ be a smooth conformal map. Then for any co-closed $\omega \in \Omega^\frac{d}{2}(\N)$, $f^*\omega$ is co-closed.
The same result holds locally, that is we do not need $\omega$ to be defined on all $\N$; pullback of any co-closed form defined on an open subset of $\N$ is co-closed.
\end{theorem}


Since any map pulls back closed forms into closed forms, we immediately obtain the following corollary:

\begin{theorem}
\label{thm:conformal_pull_back_pres}
Let $\M,\N$ be smooth oriented $d$-dimensional Riemannian manifolds, and
let $f:\M \to \N$ be a smooth conformal map. Then for any closed and co-closed $\omega \in \Omega^\frac{d}{2}(\N)$,  $f^*\omega$ is closed and co-closed.
\end{theorem}

Now, we would like to use an approximation argument, in order to obtain a weak version of this theorem.  However, as we shall immediately see, we will have to restrict the class of closed and co-closed forms we pullback to some very specific subclass.


%

The weak version of Theorem \ref{thm:conformal_pull_back_pres} we are aiming at would be something like the following:

\begin{quote}
\label{thm:conformal_pull_back_pres_weak}
Let $f \in W^{1,\frac{d}{2}}(\M,\N)$ be a weakly conformal map which is approximable by smooth maps. 

Then for any closed and co-closed $\omega \in \Omega^\frac{d}{2}(\N)$,  $f^*\omega$ is weakly closed and weakly co-closed. In particular $  f^*  \omega$ is smooth.
\end{quote}

A crucial step in obtaining such a weak version from the smooth one is to establish \emph{universal properties}, which hold for \emph{any} Sobolev map, not just conformal ones. The reason is that when we use an approximation argument, we cannot assume the approximating sequence is composed of conformal maps. Thus, we first prove some universal properties, and then we "plug in" the conformality property.

The first step is to prove a weak analogue of the fact that all smooth maps pull back closed forms to closed forms:
\begin{lemma}
\label{lem:sobolev_pullback_preserve_closedness_constant}
Let $n,d \in \mathbb{N}$, and let $\Omega \subseteq \R^n$ be  open.
Let $f \in W^{1,s}(\Omega,\R^d)$, $s \ge k\in \mathbb{N}$. Let $\omega \in \Omega^k(\R^d)$ be a \emph{constant} $k$-form; that is $\omega_q= \alpha$ independently of $q \in \R^d$, where $\alpha$ is a fixed element in $\bigwedge^k (\R^d)^* $. Then $f^*\omega$ is weakly closed.
\end{lemma}

\lemref{lem:sobolev_pullback_preserve_closedness_constant} follows from the following 

\begin{lemma}[Sobolev approximation lifts to $L^p$ convergence of the exterior powers]
\label{lem:Sobolev_approximation_lifts_to_exterior_powers}

Suppose that $f_n \in W^{1,s}(\M,\N)$ converges to $f$ in $W^{1,s}$, and that $s \ge k \in \mathbb{N}$.

Then $\bigwedge^k df_n $ converges to $\bigwedge^k df $ in $L^1$, i.e. 
\beq
\label{eq:ext_power_approx}
 \int_{\Omega}\left|\bigwedge^k df_n -\bigwedge^k df \right| \to 0.
 \eeq

\begin{comment}
To make sense of equation \eqref{eq:ext_power_approx} when the target is not \emph{Euclidean} $\R^d$, one needs to embed $\N$ isometrically into a higher-dimensional Euclidean space $\R^D$. (Since $df,df_n$ are linear maps between different vector spaces, their difference is not defined without extrinsic embedding).
\end{comment}
\end{lemma}

The proof of \lemref{lem:Sobolev_approximation_lifts_to_exterior_powers} is routine so we omit it. (In \cite{Iwa93}, the authors did not even mention explicitly this middle-step, but instead remarked that the weak analogue follows from a standard approximation argument).

\begin{proof}[Of \lemref{lem:sobolev_pullback_preserve_closedness_constant}]
Let $f_n \in C^{\infty}(\Omega,\R^d)$ satisfy $f_n \to f$ in $W^{1,s}(\Omega,\R^d)$. We need to show that for every compactly-supported $k+1$-form $\sigma \in \Omega^{k+1}(\Omega)$ 
\[
\int_{\Omega} \IP{f^*  \omega}{\de \sig}=0.
\]
Indeed,
\[
\begin{split}
\int_{\Omega} \IP{f^*  \omega}{\de \sig}=\lim_{n \to \infty} \int_{\Omega} \IP{ f_n^*  \omega}{\de \sig}=\lim_{n \to \infty} \int_{\Omega} \IP{d f_n^*  \omega}{ \sig}=0,
\end{split}
\]
where in the last equality we used the fact that $d f_n^*  \omega=f_n^*d   \omega=0$, that is pullback by smooth maps commute with the exterior derivative. We turn to justify the first equality:
\[
\left|\int_{\Omega} \IP{f^*  \omega}{\de \sig}- \IP{ f_n^*  \omega}{\de \sig}\right|  \le \int_{\Omega} |f^*  \omega-f_n^*  \omega||\de \sig| \le \|\de \sig\|_{sup} \int_{\Omega} |f^*  \omega-f_n^*  \omega|.
\]
Since 
\beq
\label{eq:pullback_diff_est1}
|f^*  \omega-f_n^*  \omega| \le  |\alpha| \, \left|\extp^{k} df-\extp^{k} df_n\right|_{\mathrm{op}},
\eeq
we obtain
\[
\left|\int_{\Omega} \IP{f^*  \omega}{\de \sig}- \IP{ f_n^*  \omega}{\de \sig}\right| \le  |\alpha| \|\de \sig\|_{\sup} \int_{\Omega} \left|\extp^{k} df-\extp^{k} df_n\right|.
\]
The RHS tends to zero by \lemref{lem:Sobolev_approximation_lifts_to_exterior_powers}.

The crucial point here is the estimate \eqref{eq:pullback_diff_est1}; let's consider it more closely:
\beq
\label{eq:pullback_diff_est2}
\begin{split}
& |f^*  \omega-f_n^*  \omega|(p)= \\
&\left|\omega_{f(p)} \circ \extp^{k} df_p -\omega_{f_n(p)} \circ \extp^{k} (df_n)_p \right| = \\
&\left|\alpha \circ  \brk{\extp^{k} df_p-\extp^{k} (df_n)_p} \right| \le |\alpha | \cdot  \left|\extp^{k} df_p-\extp^{k} (df_n)_p\right|_{\mathrm{op}}.
\end{split}
\eeq


The key point here is the fact $\omega$ is a "constant" form - we used this in the passage to the third line. Without it, we are left with estimating $\omega_{f(p)}-\omega_{f_n(p)}$. Without uniform convergence $f_n \to f$, there is probably not much hope that this quantity would converge to zero. Of course, such a uniform convergence of a smooth sequence is possible if and only if $f$ is continuous.
\end{proof}


As an immediate consequence of \lemref{lem:sobolev_pullback_preserve_closedness_constant} we have the following:

\begin{corollary}
\label{cor:co-closed_into_closed_pullbacks}
Let $\omega \in \Omega^k(\R^d)$ be a constant form. Let $f \in W^{1,s}(\Omega,\R^d)$ for $s \ge d-k \,$. Then $ f^* (\star \omega)$ is weakly closed. 
\end{corollary}

Specializing to the conformal case, we deduce the following:

\begin{proposition}
\label{prop:conf_pres_co_closed_half}
Let $\omega \in \Omega^{\frac{d}{2}}(\R^d)$ be constant, and let $f\in W^{1,s}(\Omega,\R^d)$ for $s \ge \frac{d}{2} \,$ be weakly conformal. Then $ \star f^*  \omega$ is weakly closed, which implies that $ f^*  \omega$ is weakly co-closed.
\end{proposition}

\propref{prop:conf_pres_co_closed_half} follows from \corref{cor:co-closed_into_closed_pullbacks} , using \eqref{eq:Tad_star_conf_intemed1}.

For clarity, we spell again the argument behind the \emph{preservation of co-closed forms by conformal maps}:

$\omega$ is co-closed $\Rightarrow \star \omega$ is closed $\Rightarrow f^*(\star \omega)=\star f^* \omega$ is weakly closed $\Rightarrow f^* \omega$ is weakly co-closed.

Finally, by combining \lemref{lem:sobolev_pullback_preserve_closedness_constant}  and \propref{prop:conf_pres_co_closed_half}, we get:

\begin{proposition}
\label{prop:conf_pres_closed_co_closed_half}
Let $\omega \in \Omega^{\frac{d}{2}}(\R^d)$ be constant, and let $f\in W^{1,s}(\Omega,\R^d)$ for $s \ge \frac{d}{2} \,$ be weakly conformal. Then $  f^*  \omega$ is weakly closed and co-closed. In particular $  f^*  \omega$ is smooth.
\end{proposition}

(The smoothness follows from standard elliptic regularity results).

In  \cite{Iwa93}, The authors use \propref{prop:conf_pres_closed_co_closed_half} as follows:

They take $\omega=dy^{i_1} \wedge dy^{i_2} \wedge \dots \wedge dy^{i_k}$, where $y_i$ are the standard coordinates on $\R^{2k}$. (here $d=2k$). \propref{prop:conf_pres_closed_co_closed_half} implies
\[
f^*\omega=df^{i_1} \wedge df^{i_2} \wedge \dots \wedge df^{i_k}
\]
is smooth. At this point the authors proceed in the following way:



The  smoothness of $df^{i_1} \wedge df^{i_2} \wedge \dots \wedge df^{i_k}$ implies the smoothness of 
\beq
\label{eq:det_smooth}
(df^1 \wedge df^2 \wedge \dots \wedge df^k) \wedge ( df^{k+1} \wedge \dots \wedge df^{2k})=\det df\, \Vol_{\Omega},
\eeq
i.e., the Jacobian of $f$ is smooth. Since for a conformal map $|df|^d=d^{\frac{d}{2}}\det df$, this implies $|df|^d$ is smooth, and in particular locally integrable. Hence, $f \in W_{\text{loc}}^{1,d}$.

\subsection{Completion of the regularity proof in the Euclidean case}

In the previous subsection, we showed that every $W^{1,d/2}$ weakly conformal map is in $W_{\text{loc}}^{1,d}$. At this point, the authors of \cite{Iwa93} invoke the previous known version of Liouville's theorem, which states that every $W_{\text{loc}}^{1,d}$ weakly conformal map is smooth. 

The last step of the reduction involves a substantial loss of information. The smoothness of all the $d/2$-minors is replaced by one specific corollary---the smoothness of the Jacobian. By doing this, we revert to a non-trivial problem, whose proof is not simple. Using Theorem~\ref{thm:regularity_via_minors_unique_case}, we can now use the smoothness of the minors to deduce that $f$ is smooth, in the case where $d$ is not a multiple of $4$.

To be more precise, using the refined version Theorem \ref{thm:regularity_via_minors_unique_case_a.e._assumptions} we can deduce the following:

\begin{quote}
\emph{
If $f$ is $ W^{1,d/2}_{loc}$-weakly conformal, and either $\det df>0$ a.e. or $\det df<0$ a.e., then $f$ is smooth on an open subset of full measure.
}
\end{quote}

Note that this version of the theorem is slightly weaker than the version stated in the beginning of \secref{sec:intro_Regularity_of_weakly_conformal_maps}:

\begin{itemize}
\item We needed to assume $\det df>0$ instead of $\det df \ge 0$. Indeed, without any assumption on the rank of $f$, information on the minors does not provide any information on $f$. 
\item We only deduce smoothness on an open subset of full measure. (Up to this point; we shall extend the smoothness to the whole domain below).
\item We assumed $d$ is not a multiple of $4$.
\end{itemize}

Of course, if we somehow knew that the (smooth) minors of our map $f$ were everywhere invertible, we could use Theorem \ref{thm:regularity_via_minors_unique_case} and deduce smoothness of $f$ on the entire domain directly.


One particular case where this happens is when we have a \emph{quantitative bound} on the invertibility of $df$: 
Suppose that either $\det df \ge c$ a.e. or $\det df \le -c$ a.e. for some $c>0$. Then we conclude $f$ is smooth on all $\Omega$. Indeed, $\det df \ge c$ a.e. implies $\det \extp^k df \ge c^r$ a.e., hence (by smoothness) $\det \extp^k df \ge c^r$ everywhere, so $\extp^k df$ actually lies in $\text{GL}$ everywhere. In this situation Theorem \ref{thm:regularity_via_minors_unique_case} applies verbatim.

We stress again that the slightly weaker consequence here was due to the fact that Theorem \ref{thm:regularity_via_minors_unique_case_a.e._assumptions} gives a weaker corollary. (smoothness on an open subset of full measure instead of the whole domain). Thus, it raises the question whether or not this theorem can be strengthened. We discuss this point in \secref{sec:invertibility_assumption}; we also discuss what happens when $d$ is a multiple of $4$ (where $k,d$ are both even) in \secref{sec:k_d_even_case}.

Even without assuming that the minors are everywhere invertible, we can deduce that $f$ is smooth on the entire domain, using the known expression for $C^{\infty}$ conformal maps between Euclidean domains: (This is another form of Liouville's theorem, see \cite{IM02}, Theorem 2.3.1).

Indeed, with the same notations as above, define 
$\Omega_0=\{ x \in \Omega \, | \, \bigwedge^k df_x \in \text{GL}(\bigwedge^{k}\mathbb{R}^d) \}$. 
$\Omega_0$ is open and has full measure in $\Omega$, and we know, by Theorem \ref{thm:regularity_via_minors_unique_case} that $f|_{\Omega_0}$ is smooth. (This is essentially the proof of Theorem \ref{thm:regularity_via_minors_unique_case_a.e._assumptions} ). By Liouville's theorem, $f|_{\Omega_0}$ has the form 
\[
 f(x)=b+\alpha\frac{1}{|x-a|^\epsilon}A(x-a),
 \]
where $A$ is an orthogonal matrix, and $\epsilon \in \{0,2\}, b \in \mathbb{R}^n,\alpha \in \mathbb{R},a \in \mathbb{R}^n \setminus \Omega_0$.

We shall prove that $f$ is smooth on $\Omega$. The case where $\epsilon=0$ is trivial.
Thus, up to translations, rotations and dilations, the only non-obvious case to handle is 
\beq
\label{eq:conf_formula}
 f(x)=\frac{x}{|x|^2}, \, \, \, 0 \notin \Omega_0.
 \eeq
If $0 \notin \Omega$, \eqref{eq:conf_formula} provides a smooth representative for $f \in  W^{1,d/2}_{loc}$ on all $\Omega$. (recall $\Omega_0$ is dense in $\Omega$).

Assume by contradiction that $0 \in \Omega$. A direct calculation shows that 
\[
df_x(e_i)=\frac{e_i}{|x|^2}-\frac{2x_i x}{|x|^4}.
\]
Thus $|df_x(e_i)| \to \infty$ when $x \to 0$, which implies $|df_x| \to \infty$. In particular, since $f$ is conformal $\det df \simeq |df|^d$ tends to infinity, which contradicts the already established fact that $\det df $ is smooth on $\Omega$. (This follows directly from the smoothness of the $\frac{d}{2}$-minors on $\Omega$, see \eqref{eq:det_smooth}). Thus, $0 \notin \Omega$, and $f$ is smooth.

\section{Regularity of conformal maps between Riemannian manifolds}

\label{sec:Regularit_conformalRiemannianmani}

In this section we shall prove Theorem \ref{thm:reg_conf_Riemann}. 

As mentioned in the introduction (Section \ref{subsec:intro_conf_reg_Riemm}), we shall rely in a crucial manner on the fact that every Riemannian metric locally admits a frame of closed and co-closed forms. Since the proof of this result is technical (and also in a different spirit from the rest of this paper), we shall delay it to Appendix \ref{section:Local_existence_closed_co-closed_frames}.

Similarly to the Euclidean case, we will prove that pullbacks of certain closed and co-closed forms by continuous weakly conformal maps are also weakly closed and co-closed. We shall again need to approximate our weakly conformal map by smooth maps. Unlike the Euclidean case, it is not true that every Sobolev map between manifolds can be approximated by smooth maps when $p < d$. (see \cite{Haj09}). However, every continuous Sobolev map is approximable by smooth maps (see e.g. \cite{Haj09}, Prop 2.2).

We now turn to state and prove the analogue of \lemref{lem:sobolev_pullback_preserve_closedness_constant} for maps between manifolds. Note that on an arbitrary manifold, there is no notion of "constant" form. As mentioned at the end of the proof of \lemref{lem:sobolev_pullback_preserve_closedness_constant}, we need to estimate $\omega_{f(p)}-\omega_{f_n(p)}$. This is where uniform convergence comes into play.

\begin{lemma}
\label{lem:sobolev_pullback_preserve_closedness_uniform_Riemann}
Let $\M$ be a Riemannian manifold, $D \in \mathbb{N}$, and let $f \in W^{1,s}(\M,\R^D)  \cap C^0(\M,\R^D)$, $s \ge k\in \mathbb{N}$. Let $\omega \in \Omega^k(\R^D)$ be closed. Then $f^*\omega$ is  weakly closed. 
\end{lemma}
\begin{proof}
First, we note that being weakly closed is a local property (this follows by a partition of unity argument). Thus, we can restrict the codomain to be an arbitrary open subset (this will force restriction of the domain also; this is always possible due the continuity of $f$).

\emph{In particular, we can assume without loss of generality that $\omega$ has finite $C^1$ norm (i.e. $|T\omega|_{sup},|\omega|_{sup} < \infty$). }

Since $f$ is continuous, there exist an approximating sequence $f_n \in C^{\infty}(\M,\R^D)$, $f_n \to f$ in $W^{1,s}(\M,\R^D)$, which converges uniformly to $f$. The rest of the proof is identical to the proof of \lemref{lem:sobolev_pullback_preserve_closedness_constant}, except that we need to handle Estimate \eqref{eq:pullback_diff_est1} with more care (see \eqref{eq:pullback_diff_est2} to understand the subtlety):
\beq
\label{eq:pullback_diff_est3}
\begin{split}
& |f^*  \omega-f_n^*  \omega|(p)= \\
&\left|\omega_{f(p)} \circ \extp^{k} df_p -\omega_{f_n(p)} \circ \extp^{k} (df_n)_p \right| \le \\
&\left|\brk{\omega_{f(p)}-\omega_{f_n(p)}} \circ \extp^{k} df_p +\omega_{f_n(p)} \circ \brk{ \extp^{k} df_p-\extp^{k} (df_n)_p} \right| \le \\
&|\omega_{f(p)}-\omega_{f_n(p)}| \, \, \cdot \, \, \left| \extp^{k} df_p\right| +|\omega_{f_n(p)} | \, \, \cdot \, \, \left| \extp^{k} df_p-\extp^{k} (df_n)_p \right| \le \\
&|T\omega|_{sup} \,  \cdot  \, |f(p)-f_n(p)| \,  \cdot  \, \left| \extp^{k} df_p\right| +|\omega |_{sup} \, \, \cdot \, \, \left| \extp^{k} df_p-\extp^{k} (df_n)_p \right| \le \\
&|T\omega|_{sup} \,  \cdot  \, |f-f_n|_{sup} \,  \cdot  \, \left| \extp^{k} df_p\right| +|\omega |_{sup} \, \, \cdot \, \, \left| \extp^{k} df_p-\extp^{k} (df_n)_p \right| \stackrel{L^1}{\to} 0,
\end{split}
\eeq
where the first summand tends to zero in $L^1$  by the uniform convergence, and the second tends to zero by \lemref{lem:Sobolev_approximation_lifts_to_exterior_powers}.
\end{proof}

Even though \lemref{lem:sobolev_pullback_preserve_closedness_uniform_Riemann} is an improvement compared to \lemref{lem:sobolev_pullback_preserve_closedness_constant}, we still
need a more generalized version of it, which allows for arbitrary target manifolds. The naive problem with adapting the proof above is that between manifolds, expression like
\[
\omega_{f_n(p)} \circ \extp^{k} df_p
\]
are not defined. In order to carry out the generalization, we use a technical lemma:

\begin{lemma}
\label{lem:smooth_ext_close_form_embedd}
Let $\M$ be a Riemannian manifold. 
Let $\omega \in \Omega^k(\M)$ be a smooth closed form, and let $i:\M \to \N$ be a smooth embedding. Then we can extend $\omega$ locally to a closed form on $\N$. More precisely, let $q \in \M$;  there exist a smooth closed $k$-form $\til \omega$ defined on some open neighbourhood of $i(q) \in \N$, such that $i^* \til \omega=\omega$. 
\end{lemma}

\begin{proof}
By the constant rank theorem, we can assume $\M=\R^m,\N=\R^n,q=0$ and  $i(x_1,\dots,x_m)=(x_1,\dots,x_m,0\dots,0)$. We define $\til \omega$ as follows:
First, we define it on the $i(\M)$:
\[
\brk{ \til \omega(x_1,\dots,x_m,0,\dots,0)}_{i_1,\dots,i_k}=\brk{  \omega(x_1,\dots,x_m)}_{i_1,\dots,i_k},
\]
if $1 \le i_1 < i_2 < \dots < i_k \le m$, and $\brk{ \til \omega(x_1,\dots,x_m,0,\dots,0)}_{i_1,\dots,i_k}=0$ otherwise.
Then we extend $\til \omega$ to be constant along the directions "orthogonal" to $\M$:
\[
\til \omega(x_1,\dots,x_m,x_{m+1},\dots,x_n)=\til \omega(x_1,\dots,x_m,0,\dots,0).
\]
It is now straightforward to verify that $\til \omega$ satisfies the requirements.
\end{proof}

We are now ready to prove our generalized version of  \lemref{lem:sobolev_pullback_preserve_closedness_uniform_Riemann} between manifolds:
\begin{lemma}
\label{lem:sobolev_pullback_preserve_closedness_uni_man}
Let $\M,\N$ be Riemannian manifolds. (We allow $\dim \M \neq \dim \N$).
Let $f \in W^{1,s}(\M,\N) \cap C^0(\M,\N) $, $s \ge k\in \mathbb{N}$. Let $\omega \in \Omega^k(\N)$ be closed. Then $f^*\omega$ is weakly closed. 
\end{lemma}

\begin{proof}
Let $i:\N \to \R^D$ be a smooth isometric embedding. By 
\lemref{lem:smooth_ext_close_form_embedd}, there exists a smooth closed form $\til \omega$ on $\R^D$ such that $i^*\til \omega= \omega$. Define $\til f=i \circ f:\M \to \R^D$. Then, by the definition of $W^{1,s}(\M,\N)$, $\til f \in W^{1,s}(\M,\R^D)$. Thus, by 
\lemref{lem:sobolev_pullback_preserve_closedness_uniform_Riemann},
\[
\til f^* \til \omega=(i \circ f)^* \til \omega=f^* (i^*\til \omega)=f^*\omega
\]
is weakly closed. 
\end{proof}


At this point we have at our disposal \lemref{lem:sobolev_pullback_preserve_closedness_uni_man} which is the analogue of \lemref{lem:sobolev_pullback_preserve_closedness_constant} for maps between manifolds.

From this point we can repeat the rest of the argument, regarding the pullback property of weakly conformal maps. We spell again the idea:

$\omega$ is co-closed $\Rightarrow \star \omega$ is closed $\Rightarrow f^*(\star \omega)=\star f^* \omega$ is weakly closed $\Rightarrow f^* \omega$ is weakly co-closed.

Combining  this reasoning with \lemref{lem:sobolev_pullback_preserve_closedness_uni_man} we finally obtain the following result, which is the analogue of \propref{prop:conf_pres_closed_co_closed_half}

\begin{corollary}
\label{cor:closed_and_co-closed_pullbacks_preserve}
Let $\M,\N$ be $d$-dimensional Riemannian manifolds ($d$ even).
Let $f\in W^{1,s}(\M,\N) \cap C^0(\M,\N) \, $ for $\,  s \ge \frac{d}{2} \,$ be weakly conformal, and let $\omega \in \Omega^{\frac{d}{2}}(\N)$ be a locally-defined closed and co-closed form. Then $  f^*  \omega$ is weakly closed and co-closed. In particular $  f^*  \omega$ is smooth.
\end{corollary}

We shall now prove the following proposition: 

\begin{lemma}
\label{lem:cont_minors_Riem}
Suppose $f \in W^{1,p}(\M,\N) \cap C^0(\M,\N)$ have the following property:  for any locally-defined closed and co-closed $\omega \in \Omega^k(\N)$,  $f^*\omega$ is smooth. Then the $k$-minors of $df$ are continuous.
\end{lemma}

\begin{proof}
Let $p \in \M$, and let $x_i$ be local coordinates on $\N$ around $f(p)$.

For a given increasing multi-index $I=(i_1,\dots,i_k)$, we write
\[
dx^I=dx^{i_1} \wedge dx^{i_2} \wedge \dots \wedge dx^{i_k}.
\]
$dx^I$ is a local $k$-form on $\N$. According to Theorem \ref{thm:existence_of_harmonic_frames}, there exist a local frame $\omega^J$  for $\bigwedge^k T^*\N$ around $f(p)$, consisting of closed and co-closed forms.

Since the $\omega^J$ form a frame, we can write, in some neighbourhood $U \subseteq \N$ around $f(p)$,
\[
dx^I=a_J^I\omega^J, \, \, \text{ where } \, \, a_J^I \in C^{\infty}(U).
\] 
Let $(f^i)_{1 \le i \le d}$ be the component functions of $f$ w.r.t the coordinates $x^i$. Write $df^I:=df^{i_1} \wedge df^{i_2} \wedge \dots \wedge df^{i_k}$. Then
\beq
\label{eq:recursive_reg_eq}
df^I=f^*dx^I=(a_J^I\circ f)f^*\omega^J.
\eeq

Since $f^*\omega^J$ is smooth and $f$ is continuous, equation \eqref{eq:recursive_reg_eq} implies $df^I$ is continuous on $f^{-1}(U)$. Thus, we have shown that the map $x \to  \bigwedge^k df_x$ is continuous on $f^{-1}(U)$. Since continuity is a local property, it follows that $x \to  \bigwedge^k df_x$ is continuous on $\M$. 
\end{proof}

\begin{comment} 
Let $\M,\N$ be $d$-dimensional Riemannian manifolds ($d$ even), and let $f \in W^{1,\frac{d}{2}}(\M,\N) \cap C^0(\M,\N)$ be a weakly conformal map. By \corref{cor:closed_and_co-closed_pullbacks_preserve} $f$ satisfies the assumptions of \lemref{lem:cont_minors_Riem} for $k=\frac{d}{2}$. So, we have that
\[
(df^1 \wedge df^2 \wedge \dots \wedge df^{\frac{d}{2}}) \wedge (df^{\frac{d}{2}+1} \wedge  \dots \wedge df^{d})=\det df \, \Vol_{\Omega},
\]
is continuous. As before, conformality implies $|df|^d=d^{\frac{d}{2}}|\det df|$, so $|df|^d$ is continuous, and so locally integrable. Hence, $f \in W_{\text{loc}}^{1,d}$, which implies $f$ is smooth on the whole domain by previous results. (In this regularity the theorem is known). This proves Theorem \ref{thm:reg_conf_Riemann}. 


Since the proofs of the conformal regularity in the $W^{1,d}$-regime are not easy, we shall now show how equation \eqref{eq:recursive_reg_eq}  can be used directly to obtain smoothness of $W^{1,\frac{d}{2}}$ conformal maps on an open subset of full measure (when assuming $\det df>0$ a.e. instead of $\det df \ge 0$ a.e). We phrase this result in a slightly more general form:

\end{comment}

\begin{proposition}
\label{prop:Riem_reg_harmonic_morphism_smooth}
Let $d \in \mathcal{N}$, and fix a natural number $1\le k < d$; suppose that $k,d$ are not both even.  
Let $\M,\N$ be $d$-dimensional Riemannian manifolds. 

 Suppose $f \in W^{1,p}(\M,\N) \cap C^0(\M,\N)$ satisfies $\det df > 0$ a.e or $\det df < 0$ a.e, and also have the following property:  for any locally-defined closed and co-closed $\omega \in \Omega^k(\N)$,  $f^*\omega$ is smooth. Then $f$ is smooth on an open set of full measure in $\M$.
If $k$ is odd, the assumption on $\text{sgn} (\det df)$ can be omitted; i.e. it suffices to assume  $\det df \neq 0$ a.e.
\end{proposition}

\begin{comment}
We note again that some non-degeneracy requirement on the rank of $df$ is certainly required; if $\rank df <k $ a.e. then $f^*\omega=0$ for any $k$-form, so the other assumptions hold but the conclusion may not.
\end{comment}

\begin{proof}
Since the claim is local, we can assume  that $\M,\N$ are homeomorphic to $\R^d$; The problem is then reduced to establishing the smoothness of continuous Sobolev maps $ (\R^d,\g)\to (\R^d,\h)$. (i.e. the difference between this problem and the Euclidean one is the metric structure, not the topological one). So, we can regard $df$ as a map $\R^d \to \R^d$.

We now treat the case where $k$ is odd and $\det df \neq 0$ a.e.. The alternative case, where $k$ is even, $d$ is odd, and $\det df > 0$ a.e or $\det df < 0$ a.e is treated analogously (we are only using the possibility to "reconstruct $df$ from its $k$-minors", which these conditions ensure).

%
%
%



Recall that \lemref{lem:cont_minors_Riem} implies that the map $x \to  \bigwedge^k df_x$ is continuous on $\M$. Define 
\[
\Omega_0=\left\{ x \in \M \,  \middle|  \, \bigwedge^k df_x \in \text{GL}(\bigwedge^{k}\mathbb{R}^d) \right\}.
\]
The assumption $\det df \neq 0$ a.e. implies that $\Omega_0$ is open and has full measure in $\M$.  Furthermore, on $\Omega_0$, the continuity of the minors  $\bigwedge^k df$ implies the continuity of $df$, via composition with $\psi^{-1}$: The argument in the proof of  Theorem \ref{thm:regularity_via_minors_unique_case} implies that $\bigwedge^k df \in H=\image \psi$ everywhere, hence we can apply $\psi^{-1}$. 

As in the proof of \lemref{lem:cont_minors_Riem}, let $p \in \M$, and let $x_i$ be local coordinates on $\N$ around $f(p)$. Let $\omega^J$ be a local frame for $\bigwedge^k T^*\N$ around $f(p)$, consisting of closed and co-closed forms. The structure of equation \eqref{eq:recursive_reg_eq} calls for an inductive argument. To see this more clearly, recall that on $\Omega_0$, $\bigwedge^k df \in H=\image \psi$, hence equation \eqref{eq:recursive_reg_eq} can be amplified to the following "two-parts system":
\beq
\label{eq:recursive_reg_eq2}
df^I=(a_J^I\circ f)f^*\omega^J,  \, \, \, df=\psi^{-1}\big(\bigwedge^k df\big).
\eeq

This is a recursive loop: Since $\bigwedge^k df$ encodes all the $df^I$ in one "bundle", Equation \eqref{eq:recursive_reg_eq2} implies that the regularity of $df$ "equals" the regularity of $df^I$ which "equals" the regularity of  $f$. So, the regularity of $df$ is the same as  the regularity of $f$. This creates the loop.

In more detail, we shall prove inductively that $f \in C^m \Rightarrow f \in C^{m+1}$, by using  \eqref{eq:recursive_reg_eq2} combined with essentially the same argument presented in the proof of Theorem \ref{thm:regularity_via_minors_unique_case}. Suppose that $f \in C^m$. This implies that $\extp^k df$ is $C^m$, via equation \eqref{eq:recursive_reg_eq}. The image of the $C^m$  map
\[
\phi: x \to \bigwedge^k df_x \,\,,\,\, \phi:\Omega_0 \to \GL\big(\extp^k\R^d\big)
\]
is in fact contained in $H=\image \psi$. The argument for that is identical to the one in the proof of Theorem \ref{thm:regularity_via_minors_unique_case}. (Here we use $m \ge 0$, i.e. that $\bigwedge^k df$ is at least $C^0$). $\phi$ remains $C^m$ after restricting the codomain to $H$, i.e. the map
\[
\tilde \phi: x \to \bigwedge^k df_x \,\,,\,\, \tilde \phi:\Omega_0 \to H
\]
is $C^m$. (This follows from the fact $H$ is an embedded submanifold of $\GL(\extp^k\R^d)$. This is a variant of corollary 5.30 in \cite{Lee13}). Since $\psi:\GL(\R^d) \to H$ is a diffeomorphism, the map $\Omega_0 \to \GL(\mathbb{R}^d)$ given by $ x \mapsto \psi^{-1} \circ \tilde \phi(x)=df_x $ is $C^m$, hence $f \in C^{m+1}$. This establishes the smoothness of $f$ on $\Omega_0$.
\end{proof}
\paragraph{Open Questions}
In \corref{cor:closed_and_co-closed_pullbacks_preserve},\lemref{lem:sobolev_pullback_preserve_closedness_uniform_Riemann}, and \lemref{lem:sobolev_pullback_preserve_closedness_uni_man} we assumed $f$ is continuous.
It is natural to ask whether this continuity assumption is really necessary.
Even if it can be omitted, it does not immediately imply that it is redundant also in \propref{prop:Riem_reg_harmonic_morphism_smooth}, since its proof relied directly on the continuity of $f$.
\section{Generalization of regularity via minors: Odd $k$ }
\label{sec:sec_odd_k}

When $k$ is odd, the $k$-minors of $A \in \GL(\R^d)$ uniquely determine $A$, without any assumption on $\text{sgn}(\det A)$; due to \lemref{lem:almost_injectivity_of_exterior_map_invertible_case}, the minors of $A$ determine it up to a sign. Since $k$ is odd this determines $A$ unambiguously.

The philosophy of the "regularity via minors" result (Theorem \ref{thm:regularity_via_minors_unique_case}) was roughly the following: Whenever the minors of a map are smooth, and uniquely determine its differential, the map is smooth. Thus, it is natural to check what happens in the case where $df$ has low rank. Since in that case we do not have the sign of the determinant as a distinguishing property, we will have to restrict the discussion to odd $k$. 

As we shall see, invertibility of a linear map is not necessary for a unique determination by its $k$-minors; as we already mentioned, some non-degeneracy conditions are necessary, since all the $k$-minors of a linear map with degree less then $k$ are zero. It turns out that the right condition is rank \emph{larger} than $k$. We will show various regularity results for Sobolev maps $f$ with smooth $k$-minors, and rank larger than $k$.


\paragraph{Structure of this section} In \secref{subsec:unique_determination_by_minors} we prove the unique determination by minors, which is a linear algebraic result. In Section \ref{subsec:differential_topology_of_minors_map} we prove some differential topological properties of the minors map $A \to  \bigwedge^k A$ under the \emph{constant} rank assumption (i.e. fixed rank larger than $k$). In section \ref{subsec:Regularity_via_minors_odd_k_high_rank} we prove a regularity result for maps of constant rank. In Section \ref{subsec:obstructions_to_generalize_constant_rank}
we discuss the various obstructions for generalizing the previous results beyond the constant rank
case. In Section \ref{subsec:Regularity_via_minors_bounded_case} we give a regularity result when $f \in W^{1,\infty}$ has rank larger than $k$, relaxing the constant rank assumption.

\subsection{The minors of a linear map determine it uniquely} 
\label{subsec:unique_determination_by_minors}

Let $V,W$ be $d$-dimensional real vector spaces, and let $A,B \in \text{Hom}(V,W)$. Fix $1 \le k \le d-1$. Consider the induced maps $\extp^kA,\extp^kB :\extp^k V \to \extp^k W$. We want to characterize all the pairs $(A,B)$ which satisfy
$\extp^k A=\extp^k B \neq 0$.  

\begin{lemma}
\label{lem:equal_exterior_maps_gen_imp}
Suppose that $\extp^k A=\extp^k B \neq 0$. Then $\ker A=\ker B$  and $\image(A)=\image(B)$. 
Denote $U:=\ker A=\ker B$, $\tilde W:=\image(A)=\image(B)$, and let $\tilde A,\tilde B:V/U \to \tilde W$ be the quotient operators. Then $\extp^k \til A=\extp^k \til B$.
\end{lemma}

\begin{proof} We begin by proving that $\ker A=\ker B$. Let $v \in \ker A$ and assume by contradiction that $Bv \neq 0$. Since $\dim(\image B) \ge k$, there exist $w_1,\dots,w_{k-1} \in \image B$ such that $Bv,w_1,\dots,w_{k-1}$ are linearly independent. Write $w_i=Bv_i$ for some $v_i \in V$; Then
\[
 Bv \wedge w_1 \wedge \dots \wedge w_{k-1}=\bigwedge^k A(v \wedge v_1 \wedge \dots \wedge v_{k-1})=0,
\]
which contradicts the linear independence of $Bv,w_1,\dots,w_{k-1}$. This shows $\ker A \subseteq \ker B$. The other direction follows by symmetry.

We turn to prove that $\text{Image} \, A=\text{Image} \, B$. Let $w_1 \in \image(A)$ be non-zero. Complete $w_1$ into a basis $(w_1,\dots,w_r)$ of $\image(A)$, where $r=\rank A$. Note that $\extp^k A \neq 0$ implies $r \ge k$. Write $w_i=Av_i$ for some $v_i $. Then
\[
 Bv_1 \wedge Bv_2  \wedge  \dots \wedge Bv_k = \extp^k A(v_1 \wedge  \dots \wedge v_{k})= w_1 \wedge w_2 \wedge \dots \wedge w_k \neq 0,
 \]
 which implies $\Span (  Bv_1,Bv_2,\dots, Bv_k )=\Span(w_1,\dots,w_k)$. In particular, $w_1 \in \image(B)$. Since $w_1$ was an arbitrary element in $\image(A)$ this implies $\image(A) \subseteq \image(B)$. The other direction follows by symmetry.
%
%

\end{proof}

\begin{corollary}
\label{cor:almost_injectivity_of_exterior_map_large_rank_case}
Suppose that $\extp^k A=\extp^k B $, and that $\rank(A) > k$. Then $A=\pm B$. If $k$ is odd, then $A=B$.
\end{corollary}

\begin{proof}
By \lemref{lem:equal_exterior_maps_gen_imp}, the invertible quotient operators satisfy $\extp^k \til A=\extp^k \til B$. By \lemref{lem:almost_injectivity_of_exterior_map_invertible_case}, $\til A=\pm \til B$. (Note that $\rank(\tilde B)=\rank(\tilde A)=\rank( A) > k$ so the conditions of  \lemref{lem:almost_injectivity_of_exterior_map_invertible_case} hold).
\end{proof}

\begin{comment}
This result is optimal; we cannot require less restrictive assumptions on the ranks. When $\rank(A) < k$, $\extp^k A=0$, so the $k$-minors do not provide us any information.

When $\rank(A) = k$, there is a large freedom: The condition  $\extp^k \til A=\extp^k \til B$, is essentially a coordinate-free version of saying $\det \til A=\det \til B$. (The determinant is not really well-defined here as a number, since $\til A,\til B$ are not maps from a vector space to itself; they are maps between two different spaces, on which we did not choose preferred volume forms). Thus, given such an element $A$, we essentially have a copy of $\text{SL}_k$ of elements with the same $k$-minors.

\end{comment}

\subsection{Some differential topological properties of the minors map $A \to  \bigwedge^k A$}


We have seen that the $k$-minors of a matrix uniquely  determines it when $k$ is odd and the rank is larger than $k$. (\corref{cor:almost_injectivity_of_exterior_map_large_rank_case}). Thus, it seems natural to study the differential topological properties of the minors map, when the domain is the set of matrices of rank larger than $k$, analogously to what we did in \lemref{lem:exterior_map_is_proper} and \corref{cor:image_exterior_map_is_closed_embedded}. In other words, at this stage we might aim for a regularity theorem via minors, when the rank of the map is a.e. larger than $k$. However, for various reasons there are some barriers in the way of proving such a generalization, as we shall explain in detail in Section \ref{subsec:obstructions_to_generalize_constant_rank}.

It will turn out, however, that when restricting the discussion to a \emph{specific prescribed} rank $r >k$, everything works fine. This is why we focus in this subsection on the constant-rank case.

\label{subsec:differential_topology_of_minors_map}

Let $V$ be a $d$-dimensional real vector space, and let $1 \le k \le d-1$ be fixed. For given natural numbers $r,s,$ we define
 \[
H_{r}=\{ A \in \End( V) \mid \operatorname{rank}(A) = r \}, 
\]
and
\[
\tilde H_{s}=\{ B \in \End\big(\bigwedge^k V\big) \mid \operatorname{rank}(B) = s \}.
\]
It is well-known that $H_{r}$ is an embedded submanifold of $\End(V)$ (and similarly $\tilde H_{s}$ is  an embedded submanifold of $ \End(\bigwedge^k V)$).

It is easy to see that for an element $A \in \End(V)$,
\[
\rank (\bigwedge^kA) = \binom{\rank(A)}{k}.
\]
Indeed, this can be deduced by using SVD, which reduces the problem to the diagonal case. 

The next result generalizes \corref{cor:exterior_map_properties}.

\begin{corollary}
\label{cor:exterior_map_constant_low_rank_properties}
Let $r>k$, and define $\psi:H_r \to \tilde H_{\binom {r}{k}}$ by $\psi( A)= \extp^k A$. $\psi$ is a smooth locally injective map of constant rank, hence an immersion. When $k$ is odd is $\psi$ is injective. 
\end{corollary}

\begin{proof} \corref{cor:almost_injectivity_of_exterior_map_large_rank_case} implies that $\psi$ is locally injective. We shall show below that $\psi$ has constant rank. The fact $\psi$ is an immersion then immediately follows from the rank theorem.

First, we observe that there is a natural action on $H_{r}$ by $\GL(V) \times \GL(V)$, given by
\[
(A,B) \cdot C=ACB^{-1}.
\] 
As is well-known (and easy to prove) every two endomorphisms of rank $r$ lie in the same orbit under this action. (This is a coordinate-free way of saying that every matrix of rank $r$ is equivalent to a diagonal matrix, with the first $r$ diagonal elements $1$'s, and all the rest zeros). In other words, the action described above is transitive: All endomorphisms of a given rank $r$ form an orbit.

Now, note that there is a natural action by $\GL(V) \times \GL(V)$ on  $\End(\bigwedge^k V)$, given by
\[
(A,B) \cdot D=\extp^k A \circ D \circ \extp^k B^{-1}.
\]
It is immediate to see that our map $\psi:H_{r} \to \End(\bigwedge^k V)$ is equivariant w.r.t these actions.
Since the rank of an equivariant map is constant on orbits, we are done. (see Theorem 7.25 in \cite{Lee13}; this is essentially the chain rule).
\end{proof}

The next result generalizes \lemref{lem:exterior_map_is_proper}.

\begin{lemma}
\label{lem:exterior_map_is_proper_constant_rank}
Let $r>k$. Then $\psi:H_r \to \tilde H_{\binom {r}{k}}$ is proper. 
\end{lemma}

The next result generalizes \corref{cor:image_exterior_map_is_closed_embedded}.

\begin{corollary}
\label{cor:image_exterior_map_is_closed_embedded_constant_rank}
Let $r>k$. If $k$ is odd $\psi:H_r \to \tilde H_{\binom {r}{k}}$ is an embedding, and its image $ H:=\text{Image} (\psi)$ is a closed embedded submanifold of $\til H_{\binom {r}{k}}$.
\end{corollary}

The proof is identical to the proof of \corref{cor:image_exterior_map_is_closed_embedded}.

\begin{proof}[Of \lemref{lem:exterior_map_is_proper_constant_rank}]
The proof is essentially the same as the proof of \lemref{lem:exterior_map_is_proper}, with some natural modifications. For completeness, we provide here the full argument:

Let $K \subseteq \tilde H_{\binom {r}{k}}$ be compact, and let $A_n \in \psi^{-1}(K)$. We shall prove $A_n$ has a convergent subsequence in $\psi^{-1}(K)$. It suffices to prove $A_n$ converges in $\text{End}(V)$; indeed, if $A_n \to A$, then $\bigwedge^k A_n \to \bigwedge^k A$, and the limit $\bigwedge^k A$ must be in $K$. In particular, $\binom {r}{k}=\operatorname{rank}(\bigwedge^kA) = \binom {\operatorname{rank}(A)}{k}$, so $\operatorname{rank}(A)=r$, that is $A \in H_r$.

By using SVD, we can assume $A_n=\text{diag}(\sigma_1^n,\dots,\sigma_r^n,0,\dots,0)$  is diagonal, where the first $r$ diagonal elements are non-zero, and the last $d-r$ elements are zero. (Since the orthogonal group is compact, the isometric components surely converge after passing to a subsequence).

$\bigwedge^k A_n$ is diagonal, and its first $\binom {r}{k}$ elements are of the form $\Pi_{s=1}^k \sigma_{i_s}^n$, where all the $1 \le i_s \le r$ are distinct. So, every such product converges when $n \to \infty$ to a positive number. Indeed, $\psi(A_n)=\bigwedge^k A_n \in K \subseteq \tilde H_{\binom {r}{k}}$, so it converges (after passing to a subsequence) to an element $D \in K$. Since $\text{rank}(D)=\binom {r}{k}$, it follows that the products $\Pi_{s=1}^k \sigma_{i_s}^n$ must converge to positive numbers. (If even one of them would converge to zero instead, the rank of the limit $D$ would be too low, which is a contradiction).

Now, let $1\le i \neq j \le r$. Since $r \ge k+1$, we can choose some $1 \le i_1,\dots,i_{k-1} \le r$ all different from $i,j$. Since both products $$(\Pi_{s=1}^{k-1} \sigma_{i_s}^n)\sigma_{i}^n,(\Pi_{s=1}^{k-1} \sigma_{i_s}^n)\sigma_{j}^n$$ converge to positive numbers, so does their ratio $C_{ij}^n=\frac{\sigma_i^n}{\sigma_j^n}$.

We know that 
\[
\Pi_{s=1}^k \sigma_{s}^n=\Pi_{s=1}^k \sigma_{1}^n\frac{\sigma_s^n}{\sigma_1^n}=(\sigma_{1}^n)^k  \Pi_{s=1}^k C_{s1}^n
\]
 converges to a positive number. Since all the $C_{s1}^n$ converge to positive numbers, we deduce that $\sigma_1^n$ converges. Without loss of generality, the same holds for every $\sigma_i^n$, so $A_n$ indeed converges (and we know the limit must have the right rank).
\end{proof}

\subsection{Regularity Result}

\label{subsec:Regularity_via_minors_odd_k_high_rank}



The following theorem generalizes Theorem \ref{thm:regularity_via_minors_unique_case}.

\begin{theorem}[Regularity via minors (odd $k$)]
\label{thm:regularity_via_minors_odd_k_case}
Let $d>2$, and let $\Omega$ be an open subset of $\R^d$. Let $2 \le k \le d-1$ be a fixed odd integer and let $p \ge 1$, $r >k$. 

Let $f \in W^{1,p}(\Omega,\mathbb{R}^d)$ and suppose that $\extp^k df \in \tilde H_{\binom {r}{k}}$ is smooth. Then $f$ is smooth.

\end{theorem}

\begin{comment}
Note that the assumption $\extp^k df \in \tilde H_{\binom {r}{k}}$ immediately implies that $\rank(df)=r$ a.e., even before we deduce $f$ is smooth.
\end{comment}

\begin{proof}
The proof is essentially the same as the proof of Theorem \ref{thm:regularity_via_minors_unique_case}, with some natural modifications. For completeness, we provide the full argument here:

Recall that $\psi:H_r \to \tilde H_{\binom {r}{k}}$ is given by $\psi(A)= \extp^k A$. By \corref{cor:image_exterior_map_is_closed_embedded_constant_rank},  $\psi$ is a  smooth embedding, and its image $ H:=\text{Image} (\psi)$ is a closed embedded submanifold of $\tilde H_{\binom {r}{k}}$. Thus, $\psi:H_r \to H$ is a diffeomorphism.

The map
\[
\phi: x \to \bigwedge^k df_x \,\,,\,\, \phi:\Omega \to \til H_{\binom {r}{k}}
\]
is smooth by assumption, and its image is contained in $H$. Indeed, since $H$ is closed in $ \til H_{\binom {r}{k}} $, $\phi^{-1}(H)$ is closed in $\Omega$. Furthermore, on a set of full measure $\rank( df)=r$, and for every $x \in \Omega$ where $ \rank(df_x)=r $ we clearly have $\phi(x) \in H$. 


An alternative phrasing: $\phi$ is a smooth version of the (almost everywhere defined) map $x \to \bigwedge^k df_x$. Hence, we know that almost everywhere on $\Omega$, the image $\phi(x)$ lies in $H$ (which is the subspace of all the endomorphisms $\bigwedge^k \R^d \to \bigwedge^k \R^d $ "arising" from rank $r$-endomorphisms $\R^d \to \R^d$, via exterior powers). 

This implies $\phi^{-1}(H)$ is closed and dense in $\Omega$, hence $\phi^{-1}(H)=\Omega$.
 
Since $H$ is an embedded submanifold of $\tilde H_{\binom {r}{k}}$, $\phi$ remains smooth after restricting the codomain to $H$ (see corollary 5.30 in \cite{Lee13}). We then get that the map
\[
\tilde \phi: x \to \bigwedge^k df_x \,\,,\,\, \tilde \phi:\Omega \to H
\]
is smooth.
Finally, since $\psi:H_r \to H$ is a diffeomorphism, we deduce that the following map $\Omega \to H_r$,
\[
\Omega: x \mapsto \psi^{-1} \circ \tilde \phi(x)=df_x
\]
is smooth. (More precisely, $\psi^{-1} \circ \tilde \phi$ is smooth and coincides with $df$ almost everywhere ). This establishes the smoothness of $f$.
\end{proof}

\subsection{Obstructions to generalization beyond constant rank}

\label{subsec:obstructions_to_generalize_constant_rank}

%

Consider again Theorem \ref{thm:regularity_via_minors_odd_k_case}. We restricted the statement to the case where $f$ has \emph{constant} rank larger than $k$, rather than formulating  
something which works whenever $f$ has rank larger than $k$, but not necessarily constant.

As we already said, it seems natural to only assume that $\rank(df) > k$ a.e., or alternatively (in the smooth version) $\rank(\extp^k df)>k$; after all, the $k$-minors of a map uniquely determine it whenever its rank is larger than $k$.  

However, the devil is in the details: even though the "linear-algebraic" part of unique determination still works, the differential topological aspects of the problem change when we leave the realm of constant rank maps. In this subsection we describe "what goes wrong"; along the way we describe some phenomenon which seems to be interesting on its own.


First, some more notations: Let $V$ be a $d$-dimensional real vector space. Fix an odd $1 \le k \le d-1$. Define
\beq
\label{def:space_of_endomorphisms_larger_rank}
H_{>r}=\{ A \in \End(V) \mid \operatorname{rank}(A) > r 
\},
\eeq
and
\beq
\label{def:space_of_exterior_endomorphisms_larger_rank}
\tilde H_{>s}=\{ B \in \End\big(\bigwedge^k V\big) \mid \operatorname{rank}(B) > s \}.
\eeq

$H_{>k}$ is an open submanifold of $\End(V)$.
Consider the smooth map
\[
\psi:H_{>k} \to \End\big(\extp^k V\big) \, \,, \, \, \psi(A)=\extp^{k}A.
\]
Note that $\psi(H_{>k}) \subseteq \til H_{>k}$.
\corref{cor:almost_injectivity_of_exterior_map_large_rank_case} implies $\psi$ is injective. It is possible to prove that it is actually an immersion. 
\begin{comment}
Proving that $\psi$ is an immersion is harder than in the constant rank case: We have more than one orbit for the action of $\GL(V) \times \GL(V)$ on the domain $H_{>k}$, and by general theory all we know is that the rank of $\psi$ is constant \emph{on each orbit}. Hence we cannot use the rank theorem directly; a not so short calculation is needed. (see Appendix \ref{sec:minors_map_immersion}). 
\end{comment}

Suppose we wanted to prove the following generalization of Theorem \ref{thm:regularity_via_minors_odd_k_case}:
\begin{conjecture}[Regularity via minors (odd $k$)]
\label{thm:regularity_via_minors_odd_k_case_conjecture}
Let $d>2$, and let $\Omega$ be an open subset of $\R^d$. Let $2 \le k \le d-1$ be a fixed odd integer and let $p \ge 1$.

Let $f \in W^{1,p}(\Omega,\mathbb{R}^d)$ and suppose that $\extp^k df \in \tilde H_{>k}$ is smooth. Then $f$ is smooth.
\end{conjecture}

Following the footsteps of the proof of Theorem \ref{thm:regularity_via_minors_odd_k_case}, we consider the map
\[
\psi:H_{>k} \to  \tilde H_{>k}.
\]
We already noted that $\psi$ is an injective immersion. Recall that a crucial part of our argument used the fact that the image $ \psi(H_{r})$ was a \emph{closed} embedded submanifold of $\tilde H_{\binom {r}{k}}$.

Denote $H:=\psi(H_{>k})$. In order to imitate the proof of Theorem \ref{thm:regularity_via_minors_odd_k_case}, we need $H$ to be \emph{closed} in $\til H_{>k}$; Indeed, we now want to look at the map
\[
\phi: x \to \bigwedge^k df_x \,\,,\,\, \phi:\Omega \to \til H_{>k},
\]
and claim $\image(\phi) \subseteq H$. We know that  $\phi^{-1}(H)$ is dense in $\Omega$.
In our previous case, $H= \psi(H_{r})$ was closed in the ambient space $\tilde H_{\binom {r}{k}}$, hence $\phi^{-1}(H)$ was closed in $\Omega$.

However, in the current case, \emph{it is not true that $H=\psi(H_{>k})$ is closed in $\til H_{>k}$.} Indeed, consider the following counter-example:

Let $d=5,k=3$, and set $A_n=\text{diag}(n,\frac{1}{\sqrt{n}},\frac{1}{\sqrt{n}},\frac{1}{\sqrt{n}},\frac{1}{\sqrt{n}}) \in H_{>3}$. Then 
\[
\psi(A_n)=\bigwedge^3 A_n=\text{diag}(1,1,1,1,1,1,(\tfrac{1}{\sqrt{n}})^3,(\tfrac{1}{\sqrt{n}})^3,(\tfrac{1}{\sqrt{n}})^3,(\tfrac{1}{\sqrt{n}})^3) \in H=\psi(H_{>3}) 
\]
 converges to $D=\text{diag}(1,1,1,1,1,1,0,0,0,0) \in \tilde H_{>3}$.

However, $D \notin \psi(H_{>3})$  since it does not have the right rank:
\[
\text{rank}(D)=6 \,, \, \text{and} \psi(H_{>3})= \psi(H_4)\cup \psi(H_5) \subseteq  \tilde H_{\binom {4}{3}} \cup H_{\binom {5}{3}}=\tilde H_4 \cup \tilde H_{10}.
\]

This is not the only problem with adapting the proof to this setting; there is another obstacle for continuing the proof:

Even if we somehow knew that $\image(\phi) \subseteq H$, the next step would be to say that $\phi$  remains smooth after restricting the codomain to $H$. However, it is no longer clear this is the case. The point is that the "new" $\psi$, i.e. $\psi$ considered as a map $H_{>k} \to  \tilde H_{>k}$, is not an embedding anymore. Indeed, it is certainly not proper: our example above shows this directly; the $\psi(A_n)$ lie in the compact subset $\cup_{n \in \mathbb{N}} \psi(A_n) \cup D \subseteq \til H_{>3}$ but the $A_n$ diverge. Alternatively, we can use the fact that proper maps are closed, and we just showed the image of our map is not closed.

With a little more work, we can see that $\psi:H_{>k} \to  \tilde H_{>k}$ is really not an embedding:
Let $d=6,k=3$. Take 
\[
A=\text{diag}(1,1,1,1,1,0) \in \text{End}( V).
\]
 
 Then 
 \[
 \bigwedge^3 A =\text{diag}(1,\dots,1,0,\dots,0) \in \text{End}( \bigwedge^3 V).
 \] 
 ($10$ times the value $1$, and $10$ times the value $0$).

Let $B_r(A) \in H_{>3}$ be an open ball. Assume by contradiction that $\psi(B_r(A))$ is open in $\psi(H_{>3})$, endowed with the subspace topology from $\tilde H_{>3}$.

There exist an open set $U \subseteq \text{End}( \bigwedge^3 V)$  such that $\psi(B_r(A))=\psi(H_{>3}) \cap U$.
Define 
\[
A_n=\text{diag}(n,\tfrac{1}{\sqrt n},\tfrac{1}{\sqrt n},\tfrac{1}{\sqrt n},\tfrac{1}{\sqrt n},\tfrac{1}{\sqrt n}).
\] 
Then 
\[
\psi(A_n)=\bigwedge^3 A_n =\text{diag}(1,\dots,1,(\tfrac{1}{\sqrt{n}})^3,\dots, (\tfrac{1}{\sqrt{n}})^3) 
\]
($10$ times the value $1$, and $10$ times the value $(\frac{1}{\sqrt{n}})^3$) lies in $\psi(H_{>3}) \cap U=\psi(B_r(A))$ for sufficiently large $n$. Since $\psi$ is injective, $A_n \in B_r(A)$ for sufficiently large $n$. This is a contradiction. 
To summarize, so far we have proved that
\begin{proposition}
In general, $\psi:H_{>k} \to  \tilde H_{>k}$ is not proper, and also not an embedding. Also, its image is not closed in $ \tilde H_{>k}$.
\end{proposition}

Since $\psi:H_{>k} \to  \tilde H_{>k}$ is not an embedding, its image $H$ is not an \emph{embedded} submanifold of $\til H_{>k}$. Thus, it is not clear that smooth maps into $\til H_{>k}$ whose images lie in $H$ remain smooth after restricting the codomain to $H$.

All we know is that $H$ is an \emph{immersed} submanifold of $\til H_{>k}$. In general, it is not true that for arbitrary immersed submanifolds, restriction of the codomain preserves smoothness. Immersed submanifolds which possess this special property are called \emph{weakly embedded submanifolds} (Some authors call them \emph{initial submanifolds}, or \emph{diffeological submanifolds}).

Thus, the relevant question in our context is as follows:
\begin{quote}
Is $H=\psi(H_{>k})$ a weakly embedded submanifold  of $\tilde H_{>k}$? 
\end{quote}

Weakly embedded here means that for every manifold $Q$ and for every smooth map $h:Q \to \tilde H_{>k}$, with $h(Q)\subset H$,the associated map $h:Q\to H$ is also smooth. In other words, it is always valid to restrict the range. 

The presence of certain algebraic structure implies a submanifold is weakly embedded; e.g. every orbit of a Lie group action is weakly embedded (see Theorem 5.14 in \cite{kolar1999natural} or \cite{iglesias2012smooth}). Here, however, $H$ is the image of \emph{more than one} orbit under the natural action of $\GL(V) \times \GL(V)$.

Finally, we note that a possible attempt to remedy the situation is to consider as a range for $\psi$, the subset $\cup_{r=k+1}^d \tilde  H_{\binom {r}{k}}$, instead of $\tilde H_{>k}=\cup_{i=k+1}^{\binom{d}{k}} \til H_i$. (That is, take only the admissible ranks). In that case, the image of $\psi$ is closed in the range. However, the problem now is that the new range, which is composed of unions of matrices of \emph{different} ranks, with gaps between the ranks, is no longer a submanifold of $\End(\extp^k V)$. Thus tools from differential topology, like the inverse function theorem, are no longer available.

To summarize the problems of the map $\psi:H_{>k} \to  \tilde H_{>k}$:

\begin{itemize}
\item Its image is not closed in $\tilde H_{>k}$.
\item It is not clear whether or not the image is weakly embedded.
\end{itemize}

It turns out though that there is something which could be done; in all the "problems" we encountered, there was a divergent series in the background. It turns out that when restricting the norm of the allowed endomorphisms, everything works fine. This leads to a regularity result for maps in $W^{1,\infty}$, as we present in the next subsection.

\subsection{Regularity via minors for $f \in W^{1,\infty}$}


\label{subsec:Regularity_via_minors_bounded_case} 

We begin with the following lemma:
\begin{lemma}
Let $V$ be a $d$-dimensional real vector space, and let $2 \le k \le d-1$ be a fixed odd  integer. Let $H_{>k} , \tilde H_{>k}$ be as in \defref{def:space_of_endomorphisms_larger_rank} and \defref{def:space_of_exterior_endomorphisms_larger_rank}. Let $M>0$. Then
\[
\psi:H_{>k} \cap B_M(0) \to  \tilde H_{>k}
\]
is proper, when $B_M(0)$ is the closed ball of radius $M$ in $\End(V)$, w.r.t some norm on $\End(V)$ (all the norms on a finite-dimensional space are equivalent, so the norm does not matter).
\end{lemma}

\begin{proof}
Let $K \subseteq \tilde H_{>k}$ be compact, and let $A_n \in \psi^{-1}(K)$. Since $A_n \in B_M(0)$ is bounded, it converges after passing to a subsequence, so $A_n \to A \in \End(V)$. 

Since $\bigwedge^k A_n \to \bigwedge^k A \in K \subseteq \tilde H_{>k}$ and $\binom {\operatorname{rank}(A)}{k}=\rank(\bigwedge^kA) >k $, $\operatorname{rank}(A) >k$, i.e. $A \in H_{>k}$. Since $B_M(0)$ is closed, it follows that $A \in H_{>k} \cap B_M(0)$ and we are done.


\end{proof}

At this stage we continue as in the previous two cases: Since $\psi$ is proper 
and an injective immersion (see Appendix \ref{sec:minors_map_immersion}), it is an embedding and its image is a \emph{closed embedded} submanifold (this is again elementary topology, as in  \corref{cor:image_exterior_map_is_closed_embedded}). Thus, $\psi$ is a diffeomorphism when restricted to its image, etc. Hence, all the components required for the regularity theorem are in place.

So, we can prove the same regularity result under the assumption $|df|$ is essentially bounded, i.e. $df \in L^{\infty}(\Omega,\R^{d^2})$. Since $df \in L^{\infty}(\Omega,\R^{d^2})$ implies $f \in W^{1,\infty}$ we arrive at the following version:

\begin{theorem}[Regularity via minors (odd $k$, bounded case)]
\label{thm:regularity_via_minors_odd_k_bounded_case}
Let $d>2$, and let $\Omega$ be an open subset of $\R^d$. Let $2 \le k \le d-1$ be a fixed odd integer.

Let $f \in W^{1,\infty}(\Omega,\mathbb{R}^d)$ and suppose that $\extp^k df \in \tilde H_{>k}$ is smooth. Then $f$ is smooth.
\end{theorem}

\begin{proof}
Since this is very similar to previous versions, we only describe a sketch of the adaptation required.
Consider again the map
\[
\phi: x \to \bigwedge^k df_x \,\,,\,\, \phi:\Omega \to \til H_{>k}.
\]
Write $M=\text{ess} \,  sup  | df | < \infty$, and set $H=\psi(H_{>k} \cap B_M(0))$ where $\psi:H_{>k} \cap B_M(0) \to  \tilde H_{>k}$ is the $k$-minors map. Then $H$ is closed in $H_{>k}$. Hence $\phi^{-1}(H)$ is closed and dense in $\Omega$, so $\phi^{-1}(H)=\Omega$. The rest of the proof now continues as before.
\end{proof}

%

%

\section{Discussion}
\label{sec:sec_discuss}

\subsection{The necessity of the assumption $\bigwedge^k df \in \GL(\bigwedge^k \R^d)$ }
\label{sec:invertibility_assumption}

In the proof of Theorem \ref{thm:regularity_via_minors_unique_case}, we assumed that $\bigwedge^k df \in \GL(\bigwedge^k \R^d)$ everywhere on $\Omega$. We used this assumption when we restricted the codomain of the map $x \to \bigwedge^k df_x$ to $H$, (before composing it with $\psi^{-1}$).  Of course, the weak condition $\det df>0$ a.e. together with $\bigwedge^k df \in C^{\infty}$ do not imply $\bigwedge^k df \in \GL(\bigwedge^k \R^d)$; the rank can fall on a subset of measure zero, as the following example shows: 
\[
f:\R^2 \to \R^2 \, , \, f(x,y)=(x^2-y^2,2xy) .
\]
Then $\det df=4(x^2+y^2)$ vanishes at $(0,0)$.



This raises the following \emph{open question:}
\begin{quote}
Does Theorem \ref{thm:regularity_via_minors_unique_case}  hold after removing the assumption $\bigwedge^k df \in \GL(\bigwedge^k \R^d)$? 
\end{quote}

By Theorem \ref{thm:regularity_via_minors_unique_case_a.e._assumptions} we know that $f$ must be smooth on an open subset of full measure. The question is whether or not it must be smooth on the whole domain. If not, perhaps one could say something about the Hausdorff dimension of the singular set.



%

\subsection{The case where $k,d$ are even}
\label{sec:k_d_even_case}
Suppose that $k,d$ are both even, and that all the other assumptions of Theorem \ref{thm:regularity_via_minors_unique_case} hold.
It is not clear whether $f$ should be smooth in this case as well; let $A \in \GL^+(\mathbb{R}^d)$. Since $\extp^k A=\extp^k (-A)$ and both $A,-A \in \GL^+(\mathbb{R}^d)$, the $k$-minors cannot distinguish between a map and its additive-inverse. Theoretically, $df$ could ``zig--zag" in some non smooth fashion. It is worth noting, however, that $df$ cannot switch between a given fixed invertible matrix $A$ and its negative alone. Indeed, if the gradient of a non-affine Sobolev function takes only the values $A$ and $B$, then necessarily $A-B$ is a rank one matrix. (this is Proposition 2.1 in \cite{muller1999variational}).

\emph{Open question:}
Is there a counter example for smoothness in this case? i.e. does Theorem \ref{thm:regularity_via_minors_unique_case} still hold?  

\subsubsection{Reduction to powers of two}

We note that the general case of even $d,k$ can be reduced to values of $k$ which are powers of $2$. Indeed, let us be a little more precise with the notations:

Let $d,k$ be even, $1<k<d$. Suppose $k=2^r m$, where $m$ is odd. We shall prove that  if the $k$-minors are smooth, then so are the $2^r $-minors.

For a natural number $s$, define $\psi_s:\GL^+(\R^d) \to \GL(\extp^s \R^d)$ by $\psi_s(A)= \extp^s A$, and set $H_s=\image \psi_s$. 

Consider the map $P: H_{2^r} \to H_k$ defined by the formula
\[
\extp^{2^r} A \to \extp^{k}A.
\]
It is well defined, since if $\extp^{2^r} A=\extp^{2^r} B$ then by \lemref{lem:almost_injectivity_of_exterior_map_invertible_case} $A=\pm B$. Since $k$ is even, $\extp^{k}A=\extp^{k}B$.

Furthermore, $P$ is injective: By the same argument above, if $\extp^{k}A=\extp^{k}B$, then $A=\pm B$, so $\extp^{2^r} A=\extp^{2^r} B$.

By \lemref{lem:exterior_map_is_proper} and \corref{cor:image_exterior_map_is_closed_embedded},  $\psi_s$ is a  smooth locally injective proper map, and $ H_s$ is a closed embedded submanifold of $\GL(\extp^s\R^d)$. 



$P: H_{2^r} \to H_k$ is a smooth bijective homomorphism of equidimensional Lie groups, which is also an immersion. Thus it is a diffeomorphism.


Now, suppose $f$ has all the properties assumed in Theorem \ref{thm:regularity_via_minors_unique_case}. In particular, $\extp^k df \in \GL(\extp^k \R^d)$ is smooth. The argument in the proof of Theorem \ref{thm:regularity_via_minors_unique_case} showed that $\extp^k df \in H_k$ and that 
\[
\tilde \phi: x \to \bigwedge^k df_x \,\,,\,\, \tilde \phi:\Omega \to H_k
\]
is smooth. Composing with $P^{-1}$ we obtain that the map

\[
\tilde \phi: x \to \bigwedge^{2^r} df_x \,\,,\,\, \tilde \phi:\Omega \to H_{2^r}
\]
is smooth. Thus, if the $k$-minors for $k=2^r m$ are smooth, then so are the $2^r $-minors.



\subsubsection{The case of continuous weak derivatives}

Even though we do not have a full solution yet for the case of even $k$ and $d$, we have the following partial result: 
\begin{theorem}
\label{thm:regularity_via_minors_non_unique_case}
Let $f \in W^{1,p}(\Omega,\mathbb{R}^d)$ and assume that $\det df>0$ a.e. or $\det df<0$ a.e. and that $\extp^k df  \in \GL(\extp^k \R^d)$ is smooth. Assume in addition that the weak derivatives of $f$ are continuous. Then $f$ is smooth.
\begin{comment}
\begin{itemize}
\item Since the weak derivatives are continuous, the assumptions imply $\det df \neq 0$ everywhere.
\item Continuous weak derivatives imply $C^1$.
\end{itemize}
\end{comment}
\end{theorem}

\begin{proof}
Let $x \in \Omega$. We shall prove $f$ is smooth in a neighbourhood of $x$. Repeating the proof of Theorem \ref{thm:regularity_via_minors_unique_case}, we use the local injectivity  of $\psi$ instead of its global injectivity, which no longer holds. In other words, $\psi$ is locally invertible, and the local inverse is smooth. Since $x \to df_x$ is continuous, and in particular well-defined at every point, we know "which branch of the inverse to choose".
\end{proof}

\begin{corollary}
Every possible counter-example to Theorem~\ref{thm:regularity_via_minors_unique_case} in the case where $k,d$ are both even must have non-continuous weak derivatives.
\end{corollary}

\subsection{Explicit inversion formula}
Consider again 
\[
\psi:H_{>k} \to  \tilde H_{>k} \subseteq \End(\extp^k V) \, \,, \, \, \psi(A)=\extp^{k}A.
\]
We know this map is injective. It is natural to ask whether there exist an explicit formula for $\psi^{-1}$.

If we had such a closed form formula, we could deduce the smoothness of the original map without going through the detour of the abstract machinery of differential topology. However, it seems that such a formula is only available in some very special cases, as we now describe:

Set $k=d-1$. For $A \in H_{>d-1}=\GL(V)$, $\psi(A) \in \text{GL}(\bigwedge^{d-1}V)$ can be identified with the cofactor matrix of $A$. 

A well-known identity, relating the cofactor of a matrix to its determinant, is 
\[
(\Cof A)^T \cdot A = \det A \cdot \id.
\]
This identity implies $ \Cof(\Cof A)=(\det A)^{d-2}A$. Now, if $\Cof A=B$, then $\Cof B=(\det A)^{d-2}A$, and $\det(B)=(\det(A))^{d-1}$.

Since $k=d-1$ is odd, $\det(A)=(\det(B))^{\frac{1}{d-1}}$ is the unique $ d-1 $-th root of $\det B$. Thus,
\[ 
A=(\Cof)^{-1} B=(\det B)^{\frac{2-d}{d-1}} \Cof B  
\]
gives the formula for $\psi^{-1}$. This formula can also be adapted to the case when $k=d-1$ is even, assuming $A \in \GL^+$ (take the positive $d-1$ root of $\det B$.)

Unfortunately, trying to generalize this derivation for general $k$ hits a wall.
(although there is a partial generalization, when $k$ is relatively prime to $d$, see \cite{inversionformulawill}).

{\bfseries Acknowledgements}
We thank  Eric Wofsey for suggesting the proof of \propref{prop:kernels_convergence_perturb}, and to Alex Gavrilov for suggesting the general strategy of metric's approximation for the problem of building harmonic frames.

We are also grateful to Deane Yang, Anthony Carapetis, Yael Karshon, Jake Solomon, Cy Maor and Amitai Yuval for some useful comments and discussions. Finally, we thank Raz Kupferman for carefully reading this manuscript.

This research was partially supported by the Israel Science Foundation (Grant No. 1035/17), and by a grant from the Ministry of Science, Technology and Space, Israel and the Russian Foundation for Basic Research, the Russian Federation.




%
%
%

\appendix

\section{Local existence of closed and co-closed frames}

\label{section:Local_existence_closed_co-closed_frames}
In this section we prove the following theorem:
\begin{theorem}
\label{thm:existence_of_harmonic_frames}
Let $(\M,g)$ be a smooth Riemannian manifold, and let $p \in \M$. Let $1 \le k \le d$. There exist an open neighbourhood $U$ of $p$ that admits a local frame for $\bigwedge^k T^*\M$ consisting of closed and co-closed forms. 
\end{theorem}

\begin{comment}
Globally there need not exist any non-zero harmonic forms in general (there is a topological obstruction). This theorem shows that the local situation is different.
\end{comment}

\paragraph{Argument Sketch}

For the Euclidean metric this is immediate: We have the standard frame $dx^I$. Since every metric is \emph{locally} close to being Euclidean in small neighbourhoods, the idea is to use an approximation argument:

Given a Riemannian metric $g$, we denote the space of $g$-harmonic forms of degree $k$ by $H^k_{g}$.

We view $H^k_{g}$, as a subspace of $\Omega^k(\M)$ which is "changing continuously" with the metric. Suppose $g_{\ep} \to g_0$ where $g_0$ is the Euclidean metric; we shall show that $H^k_{g_{\ep}} \to H^k_{g_0}$ in a sense which will be described later. Since being a frame is an open condition (this is essentially like being an invertible matrix), this convergence will give us a frame consisting of elements of $H^k_{g_{\ep}}$ for $g_{\ep}$ sufficiently close to $g_0$.


Even though the claim is local, and the approximation idea is also inspired by a local phenomena,
the implementation of the proof is based on a combination of local and global arguments. 

Here is the reason: On a closed manifold, being closed and co-closed is equivalent to being harmonic. Moreover, the dimension of the space of harmonic forms is a \emph{finite number} which is a \emph{topological invariant of the manifold}; it does not depend on the chosen metric. 

Thus, given a family of metrics $g_{\ep} \to g_0$ on a closed manifold $\M$, 
we consider the behaviour of the finite-dimensional subspaces $H^k_{g_{\ep}}$ (all of the same dimension) as $\ep \to 0$.

That is, we look at the map $g \to H^k_{g}=\ker \Delta_g$. As we shall see, this map is continuous in some appropriate sense; this relies on a certain \emph{stability property of kernels of linear operators}. 
That is, we shall study in what sense $T_n \to T$ implies $\ker T_n \to \ker T$. 
It turns out that a crucial factor in the existence of such a stability phenomenon is the assumption that all the kernels have \emph{the same finite dimension}. The convergence of kernels does not always hold when the dimensions are not equal or infinite.

This is the reason why we need to partially use a global perspective; a purely local approach won't work here, since locally, the space of closed and co-closed forms is infinite-dimensional. (see \cite{harmolocalinfdim}).

Since every metric is only \emph{locally} close to being Euclidean, we shall need to consider extension schemes for such "locally close" metrics into global metrics, while preserving their proximity. Due to a technical reason related to the dimension of $H^k_g$, the global topological manifold that we work with is the Torus; i.e. we view a Riemannian metric in a small neighbourhood as a metric on a small patch of the torus.

We now turn to implement the ideas just described. First, we need to understand the behaviour of kernels of operators under perturbation. We shall begin by formulating the exact stability result we have in mind; this framing would then lead us to the right norms to choose.

\paragraph{Structure of this section}

In Section \ref{susbsec:stability_kernels} we study the stability of kernels of operators. In Section \ref{susbsec:sensitivity_harmonic_metric_changes} we study the sensitivity of harmonic forms to metric changes. In Section \ref{subsec:Metric_extension_and_approximation} we prove various results regarding approximation of Riemannian metrics via Euclidean metrics. In Section \ref{subsec:proof_harmonic_frames} we combine the results of all the previous sections to prove  Theorem \ref{thm:existence_of_harmonic_frames}.


\subsection{Stability of kernels of operators}

\label{susbsec:stability_kernels}

Throughout this section we shall use the following notation: $X,Y$ are real Banach spaces, and $B(X,Y)$ is the space of bounded linear operators $X \to Y$.
Given $T \in B(X,Y)$ the \emph{modulus} of $T$ is defined to be
\[
\gamma(T):=\inf \{ \,\|Tx\| \, \, | \, \, d(x,\ker T)=1 \}.
\]
It is known that if the image of $T$ is closed in $Y$, then $\gamma(T)>0$, see e.g. Theorem IV.1.6 in \cite{goldberg2006unbounded}. 


\begin{proposition}
\label{prop:kernels_convergence_perturb}
Let $T \in B(X,Y)$ have a closed image, and suppose that $T_n \in B(X,Y)$ is a sequence converging to $T$ in the operator norm. Assume further that $\dim \ker T_n=\dim \ker T< \infty$, for every $n$.  Then $\ker T_n \to \ker T$ in the following sense: There exist bases $x^1_n,\dots,x^r_n$ for $\ker T_n$ such that $x^i_n \to x^i$, and $x^1,\dots,x^r$ form a basis for $\ker T$.

\end{proposition}

\begin{comment}
This proposition is false when we drop the assumption $T$ has a closed image, even when $X=Y$ is a Hilbert space. A counter-example is given here \cite{gerwcontinuitykernel}.
\end{comment}

\begin{proof}
We begin with the following observation:
Let $(x_n)$ be a bounded sequence in $X$ such that $T_nx_n= 0$.  Then some subsequence of $(x_n)$ converges to an element of $\ker T$. 

Indeed, since $T_n\to T$ uniformly on bounded sets, $Tx_n\to 0$.  Since $\gamma(T)>0$,  and 
$ \|Tx\| \ge \gamma(T) d(x,\ker T)$
this implies that $d(x_n,\ker T)\to 0$.  We can thus choose $y_n\in\ker T$ such that $d(x_n,y_n)\to 0$.  Since $(x_n)$ is bounded, so is $(y_n)$.  Since $\ker T$ is finite-dimensional, some subsequence of $(y_n)$ converges to an element $y\in \ker T$.  The corresponding subsequence of $(x_n)$ then also converges to $y$.

Now, let $r=\dim \ker T$.  For each $n$, choose a basis $x^1_n,\dots,x^r_n$ of $\ker T_n$ which is "almost orthonormal" in the sense that for each $i$, $x^i_n$ is a unit vector and $d(x^i_n,\operatorname{span}(x^1_n,\dots,x^{i-1}_n))\geq 1$.  

By our observation above, we may pass to a subsequence and assume that for each $i$, $(x^i_n)$ converges to some $x^i\in\ker T$.  These $x^i$ will also be almost orthonormal, and so in particular will be linearly independent and thus form a basis of $\ker T$.

\begin{comment}
Here is how we can choose  a basis $x^1_n,\dots,x^r_n$ of $\ker T_n$ which is "almost orthonormal". The selection is made inductively:
Having chosen $x^1_n,\dots,x^{i-1}_n$, let $y \in \ker T_n$ be some vector of distance $1$ from their span. Modifying $y$ by an element of the span, we can assume that $y$ is a unit vector, which we then take to be $x^i_n$.
\end{comment}
\end{proof}

\subsection{Sensitivity of harmonic forms to metric changes }

\label{susbsec:sensitivity_harmonic_metric_changes}

We shall now use the abstract machinery of \propref{prop:kernels_convergence_perturb} in the context of harmonic forms on a manifold.

\begin{proposition}
\label{prop:stability_of_harmonic_forms}
Let $(\M,g_0)$ be a smooth closed Riemannian manifold. Let $g_{\ep}$ be a family of metrics on $\M$, and suppose that $g_{\ep} \to g_0$ in $C^1$ when $\ep \to 0$. Then $H_{g_{\ep}} \to H_{g_0}$ in the following sense: There exist bases for $H_{g_{\ep}}$ which converge to a basis of $H_{g_0}$ in the uniform $C^1$ norm.
\end{proposition}

\begin{proof}
Let $D$ be the subspace of smooth closed $k$-forms on $\M$. We equip $D$ with the supremum- $C^1$ norm:
\[
\| \omega \|_{C^1,sup}:=\max\{ \|\omega\|_{sup}, \|T\omega\|_{sup}  \},
\]
where all the norms are w.r.t $g_0$.

Let $\delta_{\ep}$ be the adjoint of the exterior derivative $d$ w.r.t the metric $g_{\ep}$. We consider $\delta_{\ep}$ as as family of bounded linear operators
\[
\delta_{\ep}:(D,\| \cdot \|_{C^1,sup}) \to (\Omega^{k-1}(\M),\| \cdot \|_{sup})
\]
where the target norm is again w.r.t $g_0$.

We claim that the image of each $\delta_{\ep}$ is closed in $\Omega^{k-1}(\M)$: 

Indeed, by the Hodge theorem
\[
 \delta_{\ep}(\Omega^k(\M)) \supseteq  \delta_{\ep}(D) \supseteq \delta_{\ep}\big(\text{Image}(d)\big) = \delta_{\ep}(\Omega^k(\M))=(\text{Image}(d)\oplus H_{g_{\ep}})^{\perp}
 \] 
so $ \delta_{\ep}(D)= \delta_{\ep}(\Omega^k(\M))  $ is closed in $\Omega^k(\M)$ w.r.t the $L^2$ metric induced by $g_{\ep}$ hence (since $\M$ is compact)  also w.r.t to the $L^2$ metric induced by $g_0$. This implies it is closed w.r.t the uniform norm $\|\cdot\|_{sup}$.

Also, it is easy to see that
$
\| \delta_{\ep}-\delta_{0} \|_{op} \le C(g_0) \|g_{\ep}-g_0\|_{C^1,sup}, 
$
so $\delta_{\ep} \to \delta_{0}$ when $\ep \to 0$ (we assumed $g_{\ep} \to g_0$ in $C^1$).

\begin{comment}
We need the metrics to be $C^1$-close since the operator $\delta_g$ depends on the metric $g$ and on its first derivatives; this is due to the presence of the Hodge dual operator, whose expression  in coordinates is
\[
\star (\mathrm{d}x^{i_1}\wedge...\wedge \mathrm{d}x^{i_k}) =
  \frac{\sqrt{\det g}}{(n-k)!} g^{i_1 j_1} ... g^{i_k j_k} \ \varepsilon_{j_1
  ... j_n} \ \mathrm{d}x^{j_{k+1}} \wedge ... \wedge \mathrm{d}x^{j_n}\,.
\]
\end{comment}
Thus, we have a family of operators $\delta_{\ep}$ with closed images which converge to an operator $\delta_{0}$, where all the operators in the family have finite-dimensional kernels of the same dimension. By \lemref{prop:kernels_convergence_perturb}, the kernels converge in the required sense. 

\end{proof}

\subsection{Metric Approximations and extensions }

\label{subsec:Metric_extension_and_approximation}

In order to use \propref{prop:stability_of_harmonic_forms}, we need to study some problems of extension and approximation of Riemannian metrics. In particular, we shall need metric approximations in the $C^1$ sense. The key idea is that every metric is locally $C^1$-close to being Euclidean. 

\begin{proposition}
\label{prop:approx_metric_by_flat}
Let $\M=\mathbb{T}^d$ be the $d$-dimensional torus, and let $p \in \M$. Let $U \subseteq \M$ be a neighbourhood of $p$, and let $g$ be a Riemannian metric on $U$. 

Then, there exists a family of neighbourhoods $U_{\epsilon} \subseteq U$ around $p$, and a corresponding family of smooth Riemannian metrics $g_{\epsilon}$ on $\M$ with the following properties:

1. $g_{\epsilon}|_{U_{\epsilon}}=g|_{U_{\epsilon}}$.

2. $\| g_{\epsilon}-g_0\|_{C^1} = \mathcal{O}(\epsilon)$ everywhere on $\mathbb{T}^d$,

where $g_0$ is a metric which is isometric to the standard flat metric induced on $\mathbb{T}^d$ by the Euclidean metric on $\mathbb{R}^d$ (i.e. $g=\phi^*g_0$ where $g_0$ is the standard flat metric induced from $\mathbb{R}^d$, and $\phi:\mathbb{T}^d \to \mathbb{T}^d$ is a diffeomorphism). 


\end{proposition}

In order to prove \propref{prop:approx_metric_by_flat} we shall need the following preliminary lemmas:

\begin{lemma}
\label{lem:flat_metric_extension}
Let $\mathbb{T}^d$ be the $d$-dimensional torus. Let $p \in \M$, and let $U \subseteq \mathbb{T}^d$ be a neighbourhood of $p$. Let $g$ be a flat Riemannian metric defined on $U$.

Then, (after a possible shrinking of $U$) $g$ can be extended to a smooth metric $\tilde g$ on $\mathbb{T}^d$ which is isometric to the standard flat metric on $\mathbb{T}^d$ induced by the Euclidean metric on $\mathbb{R}^d$. 
\end{lemma}

\begin{proof}
Every flat metric is locally isometric to an Euclidean ball. So, there exists a smooth isometry $\phi:(U,g) \to (\R^d,e)$ (perhaps after a possible shrinking of $U$). If $U$ is sufficiently small, we can consider $\phi(U)$ as a subset of the flat torus, i.e. we now think of $\phi$ as an isometric embedding $\phi:(U,g) \to (\mathbb{T}^d,e)$. By a known result on extension of diffeomorphisms (\cite{palais1959natural}, Theorem 5.5), $\phi$ can be extended to a bijective diffeomorphism $\tilde \phi:\mathbb{T}^d \to \mathbb{T}^d$.

The metric $\tilde g:=\tilde \phi^* e$ is then an extension of $g$ that is isometric to the standard metric $e.$
\end{proof}

\begin{lemma}
\label{lem:metric_extension_lemma}
Let $\M$ be a smooth manifold, and let $g_0$ be a Riemannian metric on $\M$.  Let $p \in \M$, and let $U \subseteq \M$ be a neighbourhood of $p$. Suppose that $g$ is a metric on $U$, which satisfies $\| g-g_0|_{U}\|_{C^0} < \epsilon^2,\| g-g_0|_{U}\|_{C^1} < \epsilon$.

Then there exists a metric $\tilde g$ on $\M$ which coincides with $g$ on some neighbourhood of $p$, such that $\| \tilde g-g_0\|_{C^1} \le C \mathcal{o}\epsilon ( \frac{\epsilon}{\text{diam}(U)}+1)$,  where $C$ is a universal constant. 
\end{lemma}

\begin{proof}
Define $\tilde{g} = (1-\chi)g_0 + \chi g$, where $0 \le \chi \le 1$ is a smooth function that is identically $1$ on a neighbourhood of $p$ and compactly supported on $U$. 
This means that $|d\chi| \approx \frac{1}{\text{diam}(U)}$ on average, so

\beq
\label{eq:met_est1}
|d\tilde{g} - dg_0|\le   |d\chi| |g-g_0| +   |\chi||dg-dg_0| \le \frac{1}{\text{diam}(U)}\epsilon^2+\epsilon=\epsilon ( \frac{\epsilon}{\text{diam}(U)}+1).
\eeq
Since for sufficiently small $\epsilon$ we also have
\beq
\label{eq:met_est2}
|\tilde{g} - g_0|=|\chi(g-g_0)| \le |g - g_0| < \epsilon^2 < \epsilon,
\eeq
we deduce by \eqref{eq:met_est1} and \eqref{eq:met_est2} that $\| \tilde g-g_0\|_{C^1} \approx \mathcal{o}\epsilon ( \frac{\epsilon}{\text{diam}(U)}+1)$.


\end{proof}

\begin{proof}[Of \propref{prop:approx_metric_by_flat}]
Consider the expansion of the metric $g$ in normal coordinates around $p$:
\[
g_{ij} = \delta_{ij} - \frac{1}{3} R_{ikjl} \,x^kx^l + \mathcal{O}(|x|^3),
\]
which implies 
\beq
\label{eq:met_est3}
\|g_{ij}-\delta_{ij}\|_{C^1}= \mathcal{O}(|x|) , \|g_{ij}-\delta_{ij}\|_{C^0}= \mathcal{O}(|x|^2) ,
\eeq
 (where $|x|$ is measured w.r.t $g$).


By \lemref{lem:flat_metric_extension}, $\delta_{ij}$ can be extended to a smooth metric $g_0$ on $\mathbb{T}^d$ which is isometric to the standard flat metric on $\mathbb{T}^d$. (It is possible that we will have to shrink $U$ before we extend).

Let $B_{\epsilon}(p)$ be the $\epsilon$-ball around $p$ w.r.t. $g$. Then, \eqref{eq:met_est3} implies
\beq
\label{eq:met_est4}
\|g-g_0|_{B_{\epsilon}(p)}\|_{C^1}= \mathcal{O}(\epsilon) , \|g-g_0|_{B_{\epsilon}(p)}\|_{C^0}= \mathcal{O}(\epsilon^2) ,
\eeq

so by \lemref{lem:metric_extension_lemma}, there exists a metric $g_{\epsilon}$ on $\M$ which coincides with $g$ on some neighbourhood $U_{\epsilon}$ of $p$, such that $\| g_{\epsilon}-g_0\|_{C^1} =\mathcal{O}(\epsilon)$.
\end{proof}

\subsection{Proof of Theorem \ref{thm:existence_of_harmonic_frames}}

\label{subsec:proof_harmonic_frames}
%

\begin{proof} Since the claim is local, we can assume we just have a Riemannian metric on some open neighbourhood $U \subset \mathbb{T}^d$ around a given point $p \in \mathbb{T}^d$.

By \propref{prop:approx_metric_by_flat} there exist neighbourhoods $U_{\epsilon} \subseteq U$ of $p$, and metrics $g_{\epsilon}$ on $\mathbb{T}^d$ such that $g_{\epsilon}|_{U_{\epsilon}}=g|_{U_{\epsilon}}$ and $\| g_{\epsilon}-g_0\|_{C^1} = \mathcal{O}(\epsilon)$ everywhere on $\mathbb{T}^d$. (where $g_0$ is a metric which is isometric to the standard flat metric induced on $\mathbb{T}^d$ by the Euclidean metric on $\mathbb{R}^d$.)

Thus, we have $g_{\ep} \to g_0$ in $C^1$ when $\ep \to 0$. By \propref{prop:stability_of_harmonic_forms}, there exist bases $\omega^1_{\ep},\dots,\omega^s_{\ep}$ for $H_{g_{\ep}}$ which converge to a basis $\omega^1,\dots,\omega^s$ of $H_{g_0}$ in the uniform $C^1$ norm. In particular $\omega^i_{\ep}(p) \to \omega^i(p)$. 
We now claim that the $\omega^i(p)$ form a basis for $\bigwedge^k T_p^*\mathbb{T}^d$:

Consider the evaluation map $H_{g_0} \to \bigwedge^k T_p^*\mathbb{T}^d$ at $p$:
  This is a linear surjective map (Since the projected $dx^I$ from $\R^d$ are a global basis for  $H_{g_0}$). The space of harmonic $k$ forms is isomorphic to the real $k$-th cohomology; in the case of the torus, $H_{g_0} \cong H^k(\mathbb{T}^d,\mathbb{R})$ has the same dimension as $\bigwedge^k T_p^*\mathbb{T}^d$.  So, the evaluation map at $p$ is an isomorphism; in particular it maps a basis into a basis.
  
Since "being a basis" is an open condition, it follows that the $\omega^i_{\ep}(p)$ form a basis for $\bigwedge^k T_p^*\mathbb{T}^d$ for sufficiently small $\ep$. Furthermore, they stay a basis in a neighbourhood of $p$, thus they form a local frame around $p$, as required.
\end{proof}

\section{Linear algebra miscellany}
\label{sec:appendix2}

\begin{lemma}
\label{lem:k_invariant-subspaces_identity}
Let $V$ be a $d$-dimensional real vector space. Let $S:V \to V$ be a linear map, and let $1 \le k \le d-1$ be fixed. Suppose that every $k$-dimensional subspace is $S$-invariant. Then $S$ is a multiple of the identity.
\end{lemma}

\begin{proof}
Let $v\in V$ and let $X(v)$ be the collection of $k$-dimensional subspaces of $V$ that contain $v$. Then 
\[
\langle v\rangle=\bigcap_{W\in X(v)}W,
\]
and since each $W$ is $S$-invariant, $\langle v\rangle$ is also $S$-invariant, reducing the problem to the well known-case $k=1$. For completeness, we provide here the proof of this case:
Let $v,w$ be linearly independent. Then $Sv=\lambda_v v, Sw=\lambda_w w, S(v+w)=\lambda_{v+w} (v+w)$. Thus
\[
\lambda_v v+\lambda_w w=\lambda_{v+w} v+\lambda_{v+w} w, 
\]
which implies
$\lambda_v=\lambda_w$.
\end{proof}

\section{A proof $\psi:H_{>k} \to \til H_{>k}$ is an immersion}
\label{sec:minors_map_immersion}

Let $V$ be a $d$-dimensional real vector space. Fix an odd $1 \le k \le d-1$. Define
$H_{>k}=\{ A \in \End(V) \mid \operatorname{rank}(A) > k 
\}$. $H_{>k}$ is an open submanifold of $\End(V)$.
Consider the smooth map
\[
\psi:H_{>k} \to \End(\Lambda_k(V)) \, \,, \, \, \psi(A)=\extp^{k}A,
\]
\corref{cor:almost_injectivity_of_exterior_map_large_rank_case} implies $\psi$ is injective. We will shows $\psi$ is an immersion:

First, we observe that there is a natural action on $H_{>k}$ by $\GL(V) \times \GL(V)$, given by
\[
(A,B) \cdot C=ACB^{-1}.
\] 
As is well known (and easy to prove) every two endomorphism of rank $r$ are lie in the same orbit under this action. (This is a coordinate-free way of saying that every matrix of rank $r$ is equivalent to a diagonal matrix, with the first $r$ diagonal elements are $1$'s, and all the rest are zeros). In other words, the action described above has a finite number of orbits, corresponding to the different ranks: All endomorphisms of a given rank $r$ form an orbit.

We now also note that there is a natural $\GL(V) \times \GL(V)$ on  $\End(\Lambda_k(V))$, given by
\[
(A,B) \cdot D=\extp^k A \circ D \circ \extp^k B^{-1}.
\]
It is immediate to see that our map $\psi:H_{>k} \to \End(\Lambda_k(V))$ is equivariant w.r.t these actions.
Since the rank of an equivariant map is constant on orbits, it suffices to prove that $d\psi$ is injective at some canonical representative of each orbit (rank). We prove this in the following lemma:

\begin{lemma}
Let $e_i$ be a basis of $V$, and let $r >k$.  Suppose that $A \in \End(V)$ is given by the following formula:
\[
Ae_i=e_i \, \, \, \text{for } \, \, 1 \le i \le r, \, \,  \, \, \, Ae_j=0 \, \, \text{ for } \, \, r+1 \le j \le d.
\] 
Then $d\psi_A$ is injective.
\end{lemma}

\begin{proof}
$A \in H_{>k}$. Since $H_{>k}$ is an open submanifold of $\End(V)$, $T_AH_{>k}=T_A\End(V)\simeq \End(V)$. Let $B \in \ker d\psi_A \subseteq \End(V)$. Let $1 \le i_1 < i_2 < \dots < i_k \le d $.

Then
\[
0=d\psi_A(B) (e_{i_1} \wedge  \dots \wedge e_{i_k})= \sum_{s=1}^k Ae_{i_1} \wedge  \dots \wedge Be_{i_s}\wedge \dots \wedge Ae_{i_k}
\]

If $i_k \le r$, then 
\[
0=d\psi_A(B) (e_{i_1} \wedge  \dots \wedge e_{i_k})= \sum_{s=1}^k e_{i_1} \wedge  \dots \wedge Be_{i_s}\wedge \dots \wedge e_{i_k}
\]
Write $Be_i=\sum_j B_i^je_j$. Then we have
\[
0= \sum_{s=1}^k \sum_j  B_{i_s}^j e_{i_1} \wedge  \dots \wedge e_j\wedge \dots \wedge e_{i_k}
\]

Gathering together all the coefficients of $e_{i_1} \wedge  \dots \wedge e_{i_k}$ together, we obtain $\sum_{s=1}^k B^{i_s}_{i_s}=0$. So, we showed that for every sequence $1 \le i_1 < i_2 < \dots < i_k \le r $, $\sum_{s=1}^k B^{i_s}_{i_s}=0$.  

In particular, choosing $i_j=j$, we obtain  $\sum_{s=1}^k B^{s}_{s}=0$. Since $r>k$, we can also choose $i_j=j$ for $1 \le j \le k-1$, and $i_k=k+1$, so $\sum_{s=1}^{k-1} B^{s}_{s}+B^{k+1}_{k+1}=0$. By subtracting the last two equalities, we deduce $B^{k+1}_{k+1}=B^{k}_{k}$. Repeating a variation on this argument, we conclude that $B^j_j=B^i_i$ for every $1 \le i,j \le r$, and 
$\sum_{s=1}^k B^{i_s}_{i_s}=0$ becomes $kB^1_1=0$ . This implies $B^j_j=0$ for every $1 \le j \le r$. 

Now, let $1 \le i_1 < i_2 < \dots < i_k \le r $. Consider again the equation
\[
0= d\psi_A(B) (e_{i_1} \wedge  \dots \wedge e_{i_k})=\sum_{s=1}^k \sum_j  B_{i_s}^j e_{i_1} \wedge  \dots \wedge e_j\wedge \dots \wedge e_{i_k}
\]

Let  $1 \le t \le k$.
For every $j \neq i_t$, the coefficient of $e_{i_1} \wedge  \dots \wedge e_{i_{t-1}} \wedge e_j\wedge e_{i_{t+1}} \wedge\dots \wedge e_{i_k}$ is $B_{i_t}^j$. Indeed, we are in fact looking at the relevant coefficient of this element: $ \sum_{s=1}^k e_{i_1} \wedge  \dots \wedge Be_{i_s}\wedge \dots \wedge e_{i_k}$. If $s \neq t$, then the summand $e_{i_1} \wedge  \dots \wedge Be_{i_s}\wedge \dots \wedge e_{i_k}$ contains $e_{i_t}$ while $e_{i_1} \wedge  \dots \wedge e_{i_{t-1}} \wedge e_j\wedge e_{i_{t+1}} \wedge\dots \wedge e_{i_k}$ does not. Hence, we only need to see what happens for the $t$-summand 
\[
e_{i_1} \wedge  \dots \wedge Be_{i_t}\wedge \dots \wedge e_{i_k}=\sum_j  B_{i_t}^j e_{i_1} \wedge  \dots \wedge e_j\wedge \dots \wedge e_{i_k}.
\]

So, $B_{i_t}^j=0$ for every $1 \le t \le k$, and every $j \neq i_t$. Together with $B_{i_t}^{i_t}=0$, we conclude that $B_{i_t}^j=0$ for every $1 \le j \le d$ and every $1 \le t \le k$. Since $1 \le i_1 < i_2 < \dots < i_k \le r $ was arbitrary, this implies
\[
B_{s}^j=0 , \, \, \text{for every } 1 \le s \le r  \text{ and every } 1 \le j \le d.
\]

Since $Be_s=\sum_j B_s^je_j$, it follows that $Be_s=0$ for every  $1 \le s \le r$. Recall also that $Ae_s=0$ for every  $r+1 \le s \le d$.

Now, consider again
\[
0=d\psi_A(B) (e_{i_1} \wedge  \dots \wedge e_{i_k})= \sum_{s=1}^k Ae_{i_1} \wedge  \dots \wedge Be_{i_s}\wedge \dots \wedge Ae_{i_k}
\]
If we take $i_k=t$ for some $t>r$, then since $Ae_{t}=0$ all the summands except the last one vanish, and the above equality reduces to
\[
 Ae_{i_1} \wedge Ae_{i_2} \wedge  \dots \wedge  Ae_{i_{k-1}} \wedge Be_{t}=0.
 \]
Now, if we take $i_1,\dots,i_{k-1} \le r$, then this reduces to
\[
 e_{i_1} \wedge e_{i_2} \wedge  \dots \wedge  e_{i_{k-1}} \wedge Be_{t}=0,
 \]
which implies $Be_{t} \in \Span \{  e_{i_1} ,e_{i_2} ,  \dots,  e_{i_{k-1}} \}$. 

Since we can choose freely the indices $i_1,\dots,i_{k-1}$ this forces $Be_t=0$.

\end{proof}

\bibliographystyle{amsalpha}

\end{document}